\documentclass[10pt]{amsart}

\usepackage{amssymb,latexsym, mathtools,tikz}
\usepackage{enumerate}
\usepackage{graphicx}
\usepackage{float}
\usepackage{placeins}
\usepackage{mdframed}
\usepackage{amssymb}
\usepackage{esint}
\usepackage{cool}
\usepackage[all,cmtip]{xy}
\usepackage{mathtools}
\usepackage{amstext} 
\usepackage{array}   
\usepackage[shortlabels]{enumitem}
\usepackage{ytableau}

\newcolumntype{L}{>{$}l<{$}} 
\newtheorem{theorem}{Theorem}[section]
\newtheorem{lemma}[theorem]{Lemma}
\newtheorem{cor}[theorem]{Corollary}
\newtheorem{prop}[theorem]{Proposition}

\theoremstyle{definition}
\newtheorem{definition}[theorem]{Definition}
\newtheorem{example}[theorem]{Example}

\newtheorem{obs}[theorem]{Observation}
\newtheorem{notation}[theorem]{Notation}

\newtheorem{setup}[theorem]{Setup}
\theoremstyle{remark}
\newtheorem{remark}[theorem]{Remark}

\newtheorem{the context}[theorem]{The Context}
\newtheorem{question}[theorem]{Question}
\numberwithin{equation}{theorem}
\numberwithin{equation}{section}

\newcommand{\pd}{\operatorname{pd}}

\newcommand{\rank}{\operatorname{rank}}

\newcommand{\Span}{\operatorname{Span}}

\newcommand{\tor}{\operatorname{Tor}}
\newcommand{\im}{\operatorname{Im}}

\newcommand{\cone}{\operatorname{Cone}}
\newcommand{\Ker}{\operatorname{Ker}}

\newcommand{\ideal}[1]{\mathfrak{#1}}
\newcommand{\m}{\ideal{m}}

\renewcommand{\geq}{\geqslant}
\renewcommand{\leq}{\leqslant}
\renewcommand{\ker}{\Ker}
\renewcommand{\hom}{\Hom}

\newcommand{\Hom}{\operatorname{Hom}}

\newcommand{\maps}[5]{\xymatrix{#1 \ar[r]^-{#3} & #2 \\
#4 \ar@{|->}[r] & #5 \\}}

\newcommand{\mfa}{\mathfrak{a}}

\def\im{\operatorname{im}}

\newcommand{\codim}{\operatorname{codim}}
\newcommand{\tm}{\operatorname{tm}}

\begin{document}
\title[DG Structure on Length 3 Trimming Complexes]{DG Structure on Length 3 Trimming Complexes and Applications to Tor Algebras}

\author{Keller VandeBogert }
\date{\today}

\keywords{Trimming complexes, Tor-algebras, DG-algebras, realizability question, free resolutions}

\subjclass{13D02, 13D07, 13C13}

\maketitle

\begin{abstract}
    In this paper, we consider the iterated trimming complex associated to data yielding a complex of length $3$. We compute an explicit algebra structure in this complex in terms of the algebra structures of the associated input data. Moreover, it is shown that many of these products become trivial after descending to homology. We apply these results to the problem of realizability for Tor-algebras of grade $3$ perfect ideals, and show that under mild hypotheses, the process of ``trimming" an ideal preserves Tor-algebra class. In particular, we construct new classes of ideals in arbitrary regular local rings defining rings realizing Tor-algebra classes $G(r)$ and $H(p,q)$ for a prescribed set of homological data.
\end{abstract}

\section{Introduction}

Let $(R, \m , k)$ be a regular local ring with maximal ideal $\m$ and residue field $k$. A result of Buchsbaum and Eisenbud (see \cite{buchsbaum1977algebra}) established that any quotient ring $R/I$ with projective dimension $3$ admits the structure of an associative commutative differental graded (DG) algebra. Later, a complete classification of the multiplicative structure of the Tor algebra $\tor_\bullet^R (R/I , k)$ for quotient rings of projective dimension $3$ was established by Weyman in \cite{weyman1989structure} and Avramov, Kustin, and Miller in \cite{avramov1988poincare}.

This classification revealed a path to an important area of study, leading some to ask: is there a complete description of which Tor-algebra structures actually arise as the Tor-algebra of some quotient $R/I$ with some prescribed homological data? More precisely, let $R/I$ have a length $3$ DG-algebra minimal free resolution:
$$F_\bullet: \quad 0 \to F_3 \to F_2 \to F_1 \to R,$$
with $m = \rank_R (F_1)$, $n = \rank_R (F_3)$. Let $\overline{\cdot} := \cdot \otimes_R k$; one uses the induced Tor algebra multiplication to define
$$p = \rank_k (\overline{F_1}^2), \quad q = \rank_k (\overline{F_1} \cdot \overline{F_2}), \quad r = \rank_k \Big( \overline{F_2} \to \hom_k ( \overline{F_1} , \overline{F_3} ) \Big),$$
where $\overline{F_2} \to \hom_k ( \overline{F_1} , \overline{F_3} )$ associates an element of $\overline{F_1}$ to its homothety action on $\overline{F_2}$. Then, Avramov posed the question (see \cite[Question 3.8]{avramov2012cohomological}):
\begin{question}\label{question:realizability}
Which tuples $(m,n,p,q,r)$ are realized by the data defined above for some quotient ring $R/I$?
\end{question}
This is often referred to as the \emph{realizability question}, and Avramov gives bounds on the possible tuples that can occur in this same paper along with some conjectures on the tuples associated to certain Tor-algebra classes. One such conjecture was related to the Tor-algebra class $G$, where Avramov had posed:
\begin{center}
If $R/I$ is class $G(r)$ for some $r \geq2$, then is $R/I$ Gorenstein?
\end{center}
The counterexamples to the above question were originally constructed by Christensen and Veliche in \cite{christensen2014local} and produced on a much larger scale by Christensen, Veliche, and Weyman using a remarkably simple construction. Given an $\m$-primary ideal $I = (\phi_1 , \dots , \phi_n) \subseteq R$, one can ``trim" the ideal $I$ by, for instance, forming the ideal $(\phi_1 , \dots , \phi_{n-1} ) + \m \phi_n$. Under sufficient hypotheses on $I$, this process will yield an ideal that defines a \emph{non-Gorenstein} quotient ring which is also of class $G(r)$ (see \cite{christensen2019trimming}, and also work by Faucett \cite{faucett2016expanding}, where she trims complete intersections). More generally, computational evidence suggests that the process of trimming an ideal tends to preserve Tor-algebra class, so long as the multiplication on the minimal free resolution is sufficiently ``deep" in the maximal ideal (see Section \ref{sec:ToralgCons} for more precise ways of saying this). 

In this paper, we set out to answer why this is true. In practice, there are two ways of computing multiplication in the Tor-algebra $\tor_\bullet (R/I , k)$ for a given ideal $I$. First, let $K_\bullet$ denote the Koszul complex resolving $R/\m$. Then, one can descend to the homology algebra $H_\bullet (K_\bullet \otimes R/I)$, with multiplication induced by the exterior algebra $K_\bullet$. Alternatively, one can produce an explicit DG-algebra free resolution $F_\bullet$ of $R/I$, tensor with $k$, and descend to homology with product induced by the algebra structure on $F_\bullet$. We take the latter approach in this paper. An explicit free resolution of trimmed ideals is constructed in \cite{vandebogert2020trimming}. More generally, we construct an explicit algebra structure on any complex obtained by an \emph{iterated trimming} process (see Theorem \ref{thm:DGlength3}). We are then able to show that the possible nontrivial multiplications in the Tor-algebra are rather restricted (see Corollary \ref{cor:nontrivialMults}), which will allow us to deduce the hypotheses necessary for trimming to preserve the Tor-algebra class.

In Section \ref{sec:ToralgCons}, the algebra structure constructed in Section \ref{sec:theAlgStruct} is then applied to the previously mentioned case where the ideal $I$ is obtained by trimming. We focus on ideals defining rings of Tor-algebra $G$ and $H$, and show that under very mild assumptions, trimming an ideal preserves these Tor-algebra classes; these results generalize many of the trimming results obtained in \cite{christensen2019trimming}. This allows us to construct novel examples of rings of class $G(r)$ and $H(p,q)$ in Section \ref{sec:examples} obtained as quotients of arbitrary regular local rings $(R , \m , k)$ of dimension $3$. In particular, we partially answer the realizability question posed by Avramov for rings with Tor algebra class $G$ and $H$ (see Corollary \ref{cor:GrExamples} and \ref{cor:classHpqEx}), and provide the first systematic study on how trimming affects Tor-algebra classes for \emph{arbitrary} ideals defining quotients of projective dimension $3$.

The paper is organized as follows. In Section \ref{sec:backgroundandstuff}, we set the stage with conventions and notation to be used throughout the rest of the paper along with some background. In Section \ref{sec:theAlgStruct}, an explicit product on the length $3$ iterated trimming complex is constructed. In the case that the complexes involved further admit the structure of DG-modules over each other, then this product may be made even more explicit (see Proposition \ref{prop:explicitProducts}). As corollaries, we find that only a subset of the products on the iterated trimming complex are nontrivial after descending to homology.

In Section \ref{sec:ToralgCons}, we focus on the case of trimming an ideal (in the sense of Christensen, Veliche, and Weyman \cite{christensen2019trimming}). Assuming that certain products on the minimal free resolution of an ideal sit in sufficiently high powers of the maximal ideal, we show that trimming an ideal will either preserve the Tor-algebra class or yield a Golod ring (see Lemma \ref{lem:classGr} and \ref{lem:classHpq}). In the case that the minimal presenting matrix for these quotient rings has entries in $\m^2$, the restrictions become even tighter and we can say \emph{precisely} which Tor-algebra class these new ideals will occupy (see Corollary \ref{cor:quadraticGr} and \ref{cor:quadraticHpq}). 

Finally, in Section \ref{sec:examples}, we begin to construct explicit quotient rings realizing tuples of the form $(m,n,p,q,r)$. In particular, we construct an infinite class of new examples of class $G(r)$, and can say in general that there are rings of arbitrarily large type with Tor-algebra class $G(r)$, for any $r \geq 2$ (this was previously known in the case that the ambient ring was $k[x,y,z]$, see \cite{vandebogert2019structure}). Likewise, we construct an infinite class of rings of Tor-algebra class $H(p,q)$ that are not hyperplane sections, which, combined with the process of linkage, can be used to conclusively show the existence of rings realizing many of the tuples falling within the bounds imposed by Christensen, Veliche, and Weyman in \cite{christensen2020linkage}.

\section{Background, Notation, and Conventions}\label{sec:backgroundandstuff}

In this section, we first introduce some of the notation and conventions that will be in play throughout the paper. We will introduce iterated trimming complexes (see Definition \ref{def:ittrimcx}), the algebra structure on which is the main subject of Section \ref{sec:theAlgStruct}. For a more comprehensive treatment of trimming complexes, see \cite{vandebogert2020trimming}.

\begin{notation}
Let $R$ be a commutative Noetherian ring. The notation $(F_\bullet , d_\bullet)$ will denote a complex $F_\bullet$ of $R$-modules with differentials $d_\bullet$. When no confusion may occur, $F$ or $F_\bullet$ may be written instead.

Given a complex $F_\bullet$ as above, elements of $F_n$ will often be denoted $f_n$, without specifying that $f_n \in F_n$.
\end{notation}

\begin{definition}\label{def:dga}
A \emph{differential graded algebra} (\emph{DG-algebra}) over a commutative Noetherian ring $R$ is a complex of finitely generated free $R$-modules $(F_\bullet, d_\bullet)$ with differential $d$ and with a unitary, associative multiplication $\cdot: F \otimes_R F \to F$ satisfying
\begin{enumerate}[(a)]
    \item $F_i \cdot F_j \subseteq F_{i+j}$,
    \item $d_{i+j} (f_i f_j) = d_i (f_i) f_j + (-1)^i f_i d_j (f_j)$,
    \item $f_i \cdot f_j = (-1)^{ij} f_j \cdot f_i$, and
    \item $f_i^2 = 0$ if $i$ is odd.
\end{enumerate}
Given a DG-algebra $F_\bullet$, the notation $\cdot_F$ may be used to denote the product on $F_\bullet$ for extra clarity.
\end{definition}

In the following definition, we will use a slightly different sign convention than is standard for the mapping cone; this choice will be more convenient in the proof of Theorem \ref{thm:DGlength3}.

\begin{definition}\label{def:coneDef}
Let $\phi : F_\bullet \to G_\bullet$ be a morphism of complexes of free $R$-modules. The \emph{mapping cone} of $\phi$, denoted $C:= \cone (\phi)$, is the complex with
$$C_n := F_{n-1} \oplus G_{n}, \quad \textrm{and}$$
$$d_n^C := \begin{pmatrix}
d^F_{n-1} & 0 \\
(-1)^{n-1} \phi_{n-1} & d^G_{n} \\
\end{pmatrix}.$$
\end{definition}

The following setup will be used to introduce \emph{iterated trimming complexes}. For more details and proofs of the accompanying statements see \cite{vandebogert2020trimming}. Notice that the complexes appearing below need not be minimal.

\begin{setup}\label{set:trimmingcxSetup}
Let $(R , \m , k)$ be a regular local ring. Let $I \subseteq R$ be an ideal and $(F_\bullet, d_\bullet)$ a free resolution of $R/I$. 

Write $F_1 = F_1' \oplus \Big( \bigoplus_{i=1}^t Re_0^i \Big)$, where, for each $i=1, \dotsc , t$, $e^i_0$ generates a free direct summand of $F_1$. Using the isomorphism
$$\hom_R (F_2 , F_1 ) = \hom_R (F_2,F_1') \oplus \Big( \bigoplus_{i=1}^t \hom_R (F_2 , Re^i_0) \Big)$$
write $d_2 = d_2' + d_0^1 + \cdots + d^t_0$, where $d_2' \in \hom_R (F_2,F_1')$ and $d^i_0 \in \hom_R (F_2 , Re^i_0)$.  For each $i=1, \dotsc , t$, let $\mfa_i$ denote any ideal with
$$d^i_0 (F_2) \subseteq \mfa_i e^i_0,$$
and $(G^i_\bullet , m^i_\bullet)$ be a free resolution of $R/\mfa_i$. Use the notation $K' := \im (d_1|_{F_1'} : F_1' \to R)$, $K^i_0 := \im (d_1|_{Re^i_0} : Re^i_0 \to R)$, and let $J := K' + \mfa_1 \cdot K^1_0+ \cdots + \mfa_t \cdot K_0^t$.
\end{setup}

\begin{remark}
In Setup \ref{set:trimmingcxSetup}, one may alternatively assume that $R$ is a standard graded polynomial ring over a field $k$ and all over input data is homogeneous.
\end{remark}

\begin{theorem}[\cite{vandebogert2020trimming}]\label{itres}
Adopt notation and hypotheses as in Setup \ref{set:trimmingcxSetup}. Then there exists a morphism of complexes of the following form:
\begin{equation}\label{itcomx}
\xymatrix{\cdots \ar[r]^{d_{k+1}} &  F_{k} \ar[ddd]^{\begin{pmatrix} q_{k-1}^1 \\
\vdots \\
q_{k-1}^t \\
\end{pmatrix}}\ar[r]^{d_{k}} & \cdots \ar[r]^{d_3} & F_2 \ar[rrrr]^{d_2'} \ar[ddd]^{\begin{pmatrix} q_1^1 \\
\vdots \\
q_1^t \\
\end{pmatrix}} &&&& F_1' \ar[ddd]^{d_1'} \\
&&&&&&& \\
&&&&&&& \\
\cdots \ar[r]^-{\bigoplus m^i_k} & \bigoplus_{i=1}^t G^i_{k-1} \ar[r]^-{\bigoplus m^i_{k-1}} & \cdots \ar[r]^-{\bigoplus m^i_2} & \bigoplus_{i=1}^t G^i_1 \ar[rrrr]^-{-\sum_{i=1}^t  d_1(e^i_0)m_1^i} &&&& R \\}\end{equation}
where $d_1' : F_1 \to R$ is the restriction of $d_1$ to $F_1$. Moreover, the mapping cone of \ref{itcomx} is a free resolution of $R/J$. The mapping cone of \ref{itcomx} will be denoted $T_\bullet$ with differentials $\ell_\bullet$.
\end{theorem}

\begin{definition}\label{def:ittrimcx}
The \emph{iterated trimming complex} associated to the data of Setup \ref{set:trimmingcxSetup} is the complex of Theorem \ref{itres}.
\end{definition}

\begin{notation}
Given two free modules $F$ and $G$, elements of the direct sum $F \oplus G$ will be denoted $f+g \in F \oplus G$. 
\end{notation}

\begin{remark}\label{rk:theDiffs}
Adopt notation and hypotheses as in Setup \ref{set:trimmingcxSetup}. If the complexes $F_\bullet$ and $G^i_\bullet$ are all length $3$ complexes, then the complex $T_\bullet$ of Theorem \ref{itres} has differentials:
\begingroup\allowdisplaybreaks
\begin{align*}
    \ell_1 : F_1' \oplus \Big( \bigoplus_{i=1}^t G_1^i \Big) &\to R,  \\
    f_1 + \sum_{i=1}^t g_1^i &\mapsto d_1 (f_1) -\sum_{i=1}^t m_1^i (g_1^i) d_1 (e_0^i), \\
    \ell_2 : F_2 \oplus \Big( \bigoplus_{i=1}^t G_2^i \Big) & \to F_1' \oplus \Big( \bigoplus_{i=1}^t G_1^i \Big) \\
    f_2 + \sum_{i=1}^t g_2^i &\mapsto d_2' (f_2) - \sum_{i=1}^t q_1^i (f_2) + \sum_{i=1}^t m_2^i (g_2^i), \\
    \ell_3: F_3 \oplus \Big( \bigoplus_{i=1}^t G_3^i \Big) & \to F_2 \oplus \Big( \bigoplus_{i=1}^t G_2^i \Big) \\
    f_3 + \sum_{i=1}^t g_3^i &\mapsto d_3 (f_3) + \sum_{i=1}^t q_2^i (f_3) + \sum_{i=1}^t m_3^i (g_3^i). \\
\end{align*}
\endgroup
\end{remark}

\section{Algebra Structure on Length 3 Iterated Trimming Complexes}\label{sec:theAlgStruct}

In this section, we show that if the complexes associated to the input data of Setup \ref{set:trimmingcxSetup} are length $3$ DG-algebras, then the product on the resulting iterated trimming complex of Theorem \ref{itres} may be computed in terms of the products on the aforementioned complexes. The proof of this fact is a long and rather tedious computation; moreover, in full generality, the products have certain components that are only defined implicitly. In the case that the complexes involved admit additional module structures over one another, these products may be made more explicit (see Proposition \ref{prop:explicitProducts}). However, after descending to homology, many of these products either vanish completely or become considerably more simple. This fact is made explicit in the corollaries at the end of this section, and will be taken advantage of in Section \ref{sec:ToralgCons}.

The following result is proved in Proposition $1.3$ of \cite{buchsbaum1977algebra}; for convenience, we recall the result here.

\begin{obs}\label{obs:isAssoc}
Let $(F_\bullet, d_\bullet)$ denote a length $3$ resolution of a cyclic module $M$ admitting a product satisfying axioms $(a) - (d)$ of Definition \ref{def:dga}. Then, the product is associative.
\end{obs}


\begin{theorem}\label{thm:DGlength3}
Adopt notation and hypotheses as in Setup \ref{set:trimmingcxSetup}, and assume that the complexes $F_\bullet$ and $G_\bullet^i$ ($1 \leq i \leq t$) are length $3$ DG-algebras. Then the length 3 iterated trimming complex $(T_\bullet , \ell_\bullet)$ of Theorem \ref{itres} admits the structure of an associative DG-algebra. The product $T_1 \otimes T_1 \to T_2$ has the form:
\begingroup\allowdisplaybreaks
\begin{align*}
    (1) \quad F_1' \otimes F_1' &\to F_2 \oplus \Big( \bigoplus_{i=1}^t G_2^i \Big) \\
    f_1 \cdot_T f_1' &:= f_1 \cdot_F f_1' + \sum_{i=1}^t g_2^i, \ \textrm{where} \ m_2^i (g_2^i) = q_1^i (f_1 \cdot_F f_1'), \\
    (2) \quad F_1' \otimes G_1^i &\to F_2 \oplus \Big( \bigoplus_{j=1}^t G_2^j \Big) \\
    f_1 \cdot_T g_1^i &:= m_1^i(g_1^i) e_0^i \cdot_F f_1 + \sum_{j=1}^t g_2^j, \\
    \textrm{where} \ &m_2^i (g_2^i) = \begin{cases}
        d_1(f_1) g_1^i + m_1^i (g_1^i) q_1^i (e_0^i \cdot_F f_1) & \textrm{if} \ i=j,\\
         m_1^i (g_1^i) q_1^j (e_0^i \cdot_F f_1) & \textrm{if} \ j\neq i. 
    \end{cases} \\
    (3) \quad G_1^i \otimes G_1^i & \to  F_2 \oplus \Big( \bigoplus_{j=1}^t G_2^j \Big) \\
    g_1^i \cdot_T {g'}_1^i &:= -g_1^i \cdot_{G^i} {g'}_1^i d_1(e_0^i), \\
    (4) \quad G_1^i \otimes G_1^j & \to  F_2 \oplus \Big( \bigoplus_{k=1}^t G_2^k \Big) \qquad i < j \\
    g_1^i \cdot_T g_1^j &:= m_1^i (g_1^i) m_1^j (g_1^j) e_0^i \cdot_F e_0^j + \sum_{k=1}^t g_2^k, \\
    \textrm{where} \ &m_2^k (g_2^i) = \begin{cases}
        m_1^i (g_1^i) m_1^j (g_1^j) q_1^i (e_0^i \cdot_F e_0^j) + m_1^j (g_1^j) d_1 (e_0^j) g_1^i & \textrm{if} \ k=i, \\
        m_1^i (g_1^i) m_1^j (g_1^j) q_1^j (e_0^i \cdot_F e_0^j) - m_1^i (g_1^i) d_1 (e_0^i) g_1^j & \textrm{if} \ k=j, \\
        m_1^i (g_1^i) m_1^j (g_1^j) q_1^k (e_0^i \cdot_F e_0^j) & \textrm{otherwise},
    \end{cases} \\
\end{align*}
\endgroup
where $1 \leq i,j,k \leq t$. Likewise, the product $T_1 \otimes T_2 \to T_3$ has the following form:
\begingroup\allowdisplaybreaks
\begin{align*}
    (5) \quad F_1' \otimes F_2 &\to F_3 \oplus \Big( \bigoplus_{i=1}^t G_3^i \Big) \\
    f_1 \cdot_T f_2 &:= f_1 \cdot_F f_2 + \sum_{i=1}^t g_3^i, \ \textrm{for some} \ g_3^i \in G_3^i, \\
    \textrm{where} \ &m_3^i (g_3^i) = g_2^i - \sum_{j=1}^t {g}_2^{j,i} +  q_2^i (f_1 \cdot_F f_2 ), \\
    (6) \quad F_1' \otimes G_2^i &\to F_3 \oplus \Big( \bigoplus_{j=1}^t G_3^j \Big) \\
    f_1 \cdot_T g_2^i &:= \sum_{j=1}^t g_3^j, \ \textrm{for some} \ g_3^j \in G_3^j, \\
    \textrm{where} \ &m_3^j (g_3^j ) = \begin{cases}
d_1 (f_1) g_2^i - {g'_2}^i& \textrm{if} \ i=j,  \\
-{g'_2}^j & \textrm{otherwise}. \\
\end{cases} \\
    (7) \quad G_1^i \otimes G_2^i &\to F_3 \oplus \Big( \bigoplus_{j=1}^t G_3^j \Big) \\
    g_1^i \cdot_T g_2^i &= -g_1^i \cdot_{G^i} g_2^i d_1 (e_0^i), \\
    (8) \quad G_1^i \otimes G_2^j &\to F_3 \oplus \Big( \bigoplus_{k=1}^t G_3^k \Big) \\
    g_1^i \cdot_T g_2^j &:= \sum_{k=1}^t g_3^k,  \ i \neq j, \ \textrm{for some} \ g_3^k \in G_3^k , \\
    \textrm{where} \ &m_3^k (g_3^k) = \begin{cases}
    {g'_2}^i + m_1^i (g_1^i ) d_1 (e_0^i) g_2^j  & \textrm{if} \ k=j \\
    {g_2'}^k & \textrm{otherwise}, \\
    \end{cases} \\
    (9) \quad G_1^i \otimes F_2 &\to F_3 \oplus \Big( \bigoplus_{j=1}^t G_3^j \Big) \\
    g_1^i \cdot_T f_2 &:= -m_1^i (g_1^i) e_0^i \cdot_F f_2 + \sum_{j=1}^t g_3^j, \ \textrm{for some} \ g_3^j \in G_3^j, \  \textrm{where} \  \\
   m_3^j (g_3^j) = &\begin{cases}
- g_1^i \cdot q_1^i (f_2) d_1 (e_0^i) + g_2^i + \sum_{k\neq i}^t {g'_2}^{k,i} + m_1^i (g_1^i)q_2^i (e_0^i \cdot_F f_2) & \textrm{if} \ i=j, \\
g_2^j + \sum_{k \neq i }^t {g'_2}^{k,j} + m_1^i (g_1^i)q_2^j (e_0^i \cdot_F f_2) & \textrm{otherwise}. \\
\end{cases} \\
\end{align*}
\endgroup
where $1 \leq i,j,k \leq t$.
\end{theorem}

\begin{proof}
Observe that, by Observation \ref{obs:isAssoc}, it suffices to show that the contended products satisfy the Leibniz rule. The proof thus becomes a straightforward verification of this identity, and will be split accordingly into all of the cases. The complex $T_\bullet$ has differentials given explicitly as in Remark \ref{rk:theDiffs}. For convenience, let $\pi^i : F_1 \to R$ denote the composition 
$$F_1 \xrightarrow{\textrm{projection}} Re_0^i \to R,$$
where the second map sends $e_0^i \mapsto 1$. Observe that $m_1^i \circ q_1^i = \pi^i \circ d_2 =: {d_0'}^i$. Likewise, let $p : F_1 \to F_1'$ denote the natural projection. Observe that $d_2' = p \circ d_2$. Recall throughout the proof that the sign conventions used for the mapping cone differentials come from Definition \ref{def:coneDef}. Finally, all indices will take integer values between $1$ and $t$ throughout the proof.

\textbf{Case 1:} $F_1' \otimes F_1' \to F_2 \oplus \Big( \bigoplus_{i=1}^t G_2^i \Big)$. We first need to verify the existence of each $g_2^i$; by exactness of each $G^i_\bullet$, it suffices to show that $q_1^i (f_1 \cdot_F f_1')$ is a cycle. One computes:
\begingroup\allowdisplaybreaks
\begin{align*}
    m_1^i \circ q_1^i (f_1 \cdot_F f_1') &= \pi^i \Big( d_2 (f_1 \cdot_F f_1') \Big) \\
    &= \pi^i \Big( d_1 (f_1) f_1' - d_1(f_1') f_1 \Big) \\
    &= 0, \ \textrm{since} \ \pi^i (F_1') = 0 \ \textrm{for all} \ i.
\end{align*}
\endgroup
Thus the desired $g_2^i$ exists for all $i$. It remains to verify the DG axiom:
\begingroup\allowdisplaybreaks
\begin{align*}
    \ell_2 (f_1 \cdot_T f_1') &= d_2' (f_1 \cdot_F f_1') + \sum_{i=1}^t \Big( -q_1^i (f_1 \cdot_F f_1') + m_2^i (g_2^i) \Big) \\
    &= d_1 (f_1) f_1' - d_1(f_1') f_1 \\
    &= \ell_1 (f_1) f_1' - \ell_1 (f_1') f_1. \\
\end{align*}
\endgroup
\textbf{Case 2:} $F_1' \otimes G_1^i \to F_2 \oplus \Big( \bigoplus_{j=1}^t G_2^j \Big)$. We first verify the existence of the desired $g_2^j$. One computes:
\begingroup\allowdisplaybreaks
\begin{align*}
    m_1^i \Big( d_1(f_1) g_1^i + m_1^i (g_1^i) q_1^i (e_0^i \cdot_F f_1) \Big) &= d_1(f_1) m_1^i (g_1^i) + m_1^i (g_1^i) \pi^i (d_2 (e_0^i \cdot_F f_1)) \\
    &= d_1 (f_1) m_1^i (g_1^i) - m_1^i (g_1^i) d_1 (f_1) \\
    &= 0, \\
    m_1^j \Big( m_1^i (g_1^i) q_1^j (e_0^i \cdot_F f_1) \Big) &= m_1^i (g_1^i) \pi^j (d_2( e_0^i \cdot_F f_1)) \\
    &=0. \\
\end{align*}
\endgroup
It remains to verify the DG axioms:
\begingroup\allowdisplaybreaks
\begin{align*}
    \ell_2 (f_1 \cdot_T g_1^i) &= m_1^i (g_1^i) d_2' (e_0^i \cdot_F f_1) -  \sum_{j=1}^t m_1^i (g_1^i) q_1^j (e_0^i \cdot_F f_1) + \sum_{j=1}^t m_2^i (g_2^i) \\
    &=  m_1^i (g_1^i) d_1 (e_0) f_1 + d_1 (f_1) g_1^i \\
    &= \ell_1 (f_1) g_1^i - \ell_1 (g_1^i) f_1. \\
\end{align*}
\endgroup
The above uses that $d_2 ' = p \circ d_2$, implying $d_2' (e_0^i \cdot_F f_1) = d_1(e_0^i) f_1$.

\textbf{Case 3:} $G_1^i \otimes G_1^i \to  F_2 \oplus \Big( \bigoplus_{j=1}^t G_2^j \Big)$. One computes directly:
\begingroup\allowdisplaybreaks
\begin{align*}
    \ell_2 (g_1^i \cdot_T {g'}_1^i ) &= m_2^i (- g_1^i \cdot_{G^i} {g'}_1^i d_1 (e_0^i)) \\
    &= -m_1(g_1^i) {g'}_1^i d_1 (e_0^i) + m_1({g'}_1^i ) g_1^i d_1(e_0^i) \\
    &= \ell_1 (g_1^i) {g'}_1^i - \ell_1 ({g'}_1^i) g_1^i \\
\end{align*}
\endgroup

\textbf{Case 4:} $G_1^i \otimes G_1^j \to  F_2 \oplus \Big( \bigoplus_{k=1}^t G_2^k \Big)$, $i \neq j$. We verify the existence of $g_2^i$; the proof of the existence of $g_2^j$ is identical. One computes:
\begingroup\allowdisplaybreaks
\begin{align*}
    & m_2^i \Big( m_1^i (g_1^i) m_1^j (g_1^j) q_1^i (e_0^i \cdot_F e_0^j) + m_1^j (g_1^j) d_1 (e_0^j) g_1^i \Big) \\
    = & m_1^i (g_1^i) m_1^j (g_1^j) \pi^i \circ d_2 (e_0^i \cdot_F e_0^j) + m_1^j (g_1^j) d_1 (e_0^j) m_1^i (g_1^i) \\  
    =& -m_1^i (g_1^i) m_1^j (g_1^j) d_1 (e_0^j) + m_1^j (g_1^j) d_1 (e_0^j) m_1^i (g_1^i) \\ 
    =& 0.
\end{align*}
\endgroup
It remains to show the DG axiom:
\begingroup\allowdisplaybreaks
\begin{align*}
    \ell_2 (g_1^i \cdot_T g_1^j ) &= m_1^i (g_1^i) m_1^j (g_1^j) d_2' (e_0^i \cdot_F e_0^j) \\
    &+\sum_{k=1}^t \Big( -m_1^i (g_1^i) m_1^j (g_1^j)  q_1^i (e_0^i \cdot_F e_0^j) + m_2^k (g_2^k) \Big) \\
    &= m_1^j (g_1^j) d_1 (e_0^j) g_1^i - m_1^i (g_1^i) d_1(e_0^i) g_1^j \\
    &= \ell_1 (g_1^i) g_1^j - \ell_1 (g_1^j) g_1^i. \\
\end{align*}
\endgroup
In the above, notice that $d_2' (e_0^i \cdot_F e_0^j) = p (d_1(e_0^i) e_0^j - d_1(e_0^j) e_0^i) = 0$.

 \textbf{Case 5:} $F_1' \otimes F_2 \to F_3 \oplus \Big( \bigoplus_{i=1}^t G_3^i \Big)$. 
Observe that
\begingroup\allowdisplaybreaks
\begin{align*}
    f_1 \cdot_T (d_2' (f_2) - \sum_{i=1}^t q_1^i (f_2) ) &= f_1 \cdot_F d_2' (f_2) +\sum_{i=1}^t g_2^i \\
    &- \sum_{i=1}^t \Big( m_1^i \circ q_1^i (f_2) e_0 \cdot_F f_1 + \sum_{j=1}^t {g}_2^{i,j} \Big), \\
    \textrm{where} \qquad & m_2^i \Big(  g_2^i - \sum_{j=1}^t {g}_2^{j,i}  \Big) \\
    =&   q_1^i (f_1 \cdot_F d_2' (f_2) ) - d_1(f_1) q_1^i (f_2) - \sum_{j=1}^t q_1^i ( d_0^j (f_2) \cdot_F f_1)   \\
    =&   q_1^i ( f_1 \cdot_F d_2 (f_2) - d_1 (f_1) f_2 )  \\
    =& - m_2^i \circ q_2^i (f_1 \cdot_F f_2 ), \quad \textrm{for each} \ i=1 , \dotsc , t. \\
\end{align*}
\endgroup
This implies that $ g_2^i - \sum_{j=1}^t {g}_2^{j,i} +  q_2^i (f_1 \cdot_F f_2 )$ is a cycle, so that there exist $g_3^i \in G_3^i$ such that
$$m_3^i (g_3^i) = g_2^i - \sum_{j=1}^t {g}_2^{j,i} +  q_2^i (f_1 \cdot_F f_2 ).$$
Using this along with the fact that 
$$\sum_{i=1}^t \Big(  g_2^i - \sum_{j=1}^t {g}_2^{i,j}  \Big) = \sum_{i=1}^t \Big(  g_2^i - \sum_{j=1}^t {g}_2^{j,i}  \Big),$$ 
one obtains:
\begingroup\allowdisplaybreaks
\begin{align*}
    f_1 \cdot_T (d_2' (f_2) - \sum_{i=1}^t q_1^i (f_2)) &= f_1 \cdot_F d_2 (f_2) - \sum_{i=1}^t q_2^i (f_1 \cdot_F f_2) + \sum_{i=1}^t g_3^i, \\
\end{align*}
\endgroup
whence upon choosing
$$f_1 \cdot_T f_2 := f_1 \cdot_F f_2 + \sum_{i=1}^t g_3^i,$$
one immediately obtains
\begingroup\allowdisplaybreaks
\begin{align*}
    \ell_3 (f_1 \cdot_T f_2) &= d_1 (f_1) f_2 - f_1 \cdot_F d_2 (f_2) +  \sum_{i=1}^t q_2^i (f_1 \cdot_F f_2) + \sum_{i=1}^t g_3^i \\
    &= d_1 (f_1) f_2 - f_1 \cdot_T ( d_2' (f_2) - \sum_{i=1}^t q_1^i (f_2)) \\
    &= \ell_1 (f_1) f_2 - f_1 \cdot_T \ell_2 (f_2). \\
\end{align*}
\endgroup
\textbf{Case 6:} $F_1' \otimes G_2^i \to F_3 \oplus \Big( \bigoplus_{j=1}^t G_3^j \Big)$. One computes:
\begingroup\allowdisplaybreaks
\begin{align*}
    \ell_1 (f_1) g_2^i - f_1 \cdot_T m_2^i (g_2^i) &= d_1 (f_1) g_2^i - \sum_{j=1}^t {g'}_2^j, \\
    \textrm{where} \ m_2^j ({g'}_2^j) &= \begin{cases} d_1 (f_1) m_2^i (g_2^i) & \textrm{if} \ i=j, \\
    0 & \textrm{otherwise}. \\
    \end{cases} \\
\end{align*}
\endgroup
This implies that
$$ \begin{cases}
d_1 (f_1) g_2^i - {g'_2}^i& \textrm{if} \ i=j, \ \textrm{and} \\
{g'_2}^j & \textrm{otherwise} \\
\end{cases}$$
are cycles. By exactness of each $G_\bullet^j$, there exist $g_3^j \in G_3^j$ such that
\begin{equation}\label{eq:case6}
    m_3^j (g_3^j ) = \begin{cases}
d_1 (f_1) g_2^i - {g'_2}^i& \textrm{if} \ i=j,  \\
-{g'_2}^j & \textrm{otherwise}. \\
\end{cases}
\end{equation}

\textbf{Case 7:} $G_1^i \otimes G_2^i \to F_3 \oplus \Big( \bigoplus_{j=1}^t G_3^j \Big)$. One computes:
\begingroup\allowdisplaybreaks
\begin{align*}
    \ell_3 ( g_1^i \cdot_T g_2^i ) &= m_3^i ( - g_1^i \cdot_{G^i} g_2^i d_1 (e_0^i)) \\
    &= -m_1^i (g_1^i) g_2^i d_1 (e_0^i) + g_1^i \cdot_{G^i} m_2^i (g_2^i) d_1 (e_0^i) \\
    &= \ell_1 (g_1^i) g_2^i - g_1^i \cdot_T \ell_2 (g_2^i). \\
\end{align*}
\endgroup
\textbf{Case 8:} $G_1^i \otimes G_2^j \to F_3 \oplus \Big( \bigoplus_{k=1}^t G_3^k \Big)$, $i \neq j$. One computes:
\begingroup\allowdisplaybreaks
\begin{align*}
    \ell_1 (g_1^i) g_2^j - g_1^i \cdot_T \ell_2 (g_2^j) &= -m_1^i (g_1^i) d_1 (e_0^i) g_2^j - \sum_{k=1}^t {g'}_2^k \\
    \textrm{where} \ m_2^k ({g'}_2^k) &= \begin{cases}
    - m_1^i (g_1^i ) d_1 (e_0^i) m_2^j (g_2^j) & \textrm{if} \ k=j \\
    0 & \textrm{otherwise}, \\
    \end{cases} \\
\end{align*}
\endgroup
whence
\begingroup\allowdisplaybreaks
\begin{align*}
     \begin{cases}
    {g'_2}^i + m_1^i (g_1^i ) d_1 (e_0^i) g_2^j  & \textrm{if} \ k=j \\
    {g_2'}^k & \textrm{otherwise}, \\
    \end{cases} \\
\end{align*}
\endgroup
are cycles, implying there exists $g_3^k \in G_3^k$ such that
\begin{equation}\label{eq:case8}
    m_3^k (g_3^k) = \begin{cases}
    {g'_2}^i + m_1^i (g_1^i ) d_1 (e_0^i) g_2^j  & \textrm{if} \ k=j \\
    {g_2'}^k & \textrm{otherwise}. \\
    \end{cases}
\end{equation}
Thus, one may define
$$g_1^i \cdot_T g_2^j := \sum_{k=1}^t g_3^k,$$
and this product will satisfy the Leibniz rule.

\textbf{Case 9:} $G_1^i \otimes F_2 \to F_3 \oplus \Big( \bigoplus_{j=1}^t G_3^j \Big)$. One computes:
\begingroup\allowdisplaybreaks
\begin{align*}
    &\quad \ell_1(g_1^i) f_2 - g_1^i \cdot_T \ell_2 (f_2) \\
    &= - m_1^i(g_1^i) d_1 (e_0^i) f_2 - g_1^i \cdot_T (d_2' (f_2) - \sum_{j=1}^t q_1^j (f_2) ) \\
    &= -m_1^i (g_1^i) d_1 (e_0^i ) f_2 + d_2' (f_2) \cdot_T g_1^i + \sum_{j=1}^t g_1^i \cdot_T q_1^j (f_2) \\
    &= -m_1^i (g_1^i) d_1 (e_0^i ) f_2 + m_1^i (g_1^i) e_0^i \cdot_F d_2' (f_2)  +\sum_{j=1}^t g_2^j - g_1^i \cdot_{G^i} q_1^i (f_2) d_1 (e_0^i) \\
    &+ \sum_{j \neq i }^t \Big( m_1^i (g_1^i) m_1^j (q_1^j (f_2)) e_0^i \cdot_F e_0^j + \sum_{k=1}^t {g'_2}^{j,k} \Big)  \\
    &= -m_1^i (g_1^i)d_3 ( e_0 \cdot_F f_2 ) - g_1^i \cdot_{G^i} q_1^i (f_2) d_1 (e_0^i) + g_2^i + \sum_{j\neq i}^t \Big( g_2^j + \sum_{k=1}^t {g'_2}^{j,k} \Big) . \\
\end{align*}
\endgroup
Observe that $\sum_{j\neq i}^t \Big( g_2^j + \sum_{k=1}^t {g'_2}^{j,k} \Big) = \sum_{j\neq i}^t \Big( g_2^j + \sum_{k\neq i}^t {g'_2}^{k,j} \Big) + \sum_{k\neq i}^t {g'_2}^{k,i}$, and
\begingroup\allowdisplaybreaks
\begin{align*}
    m_2^j \Big( g_2^j + \sum_{k \neq i }^t {g'_2}^{k,j} \Big) &= m_1^i (g_1^i) q_1^j (e_0^i \cdot_F d_2' (f_2)) \\
    &+ \sum_{k \neq i} m_1^i (g_1^i ) {d_0'}^k (f_2) q_1^j (e_0^i \cdot_F e_0^k)   \\
    &- m_1^i (g_1^i) d_1 (e_0^i) q_1^j (f_2) \\
    &= m_1 (g_1^i) q_1^j (e_0^i \cdot_F d_2 (f_2) ) - m_1^i (g_1^i) d_1 (e_0^i) q_1^j (f_2) \\
    &= -m_1^i (g_1^i) m_2^j (q_2^j (e_0^i \cdot_F f_2) ), \\
    m_2^i \Big( - g_1^i \cdot q_1^i (f_2) d_1 (e_0^i) + g_2^i + \sum_{k\neq i}^t {g'_2}^{k,i} \Big) &= -m_1^i (g_1^i) q_1^i (f_2) d_1 (e_0^i) \\
    &+ {d_0'}^i (f_2) d_1 (e_0^i) + d_1 (d_2' (f_2)) g_1^i \\
    &+ m_1^i (g_1^i) q_1^i (e_0^i \cdot_F d_2' (f_2)) \\
    &+ \sum_{k\neq i} \Big( m_1^i (g_1^i) {d_0'}^k (f_2) q_1^i (e_0^i \cdot_F e_0^k) + {d_0'}^k (f_2) d_1 (e_0^k) g_1^i \Big) \\
    &= d_1 (d_2 (f_2)) g_1^i -m_1^i (g_1^i) q_1^i (f_2) d_1 (e_0^i) \\
    &+  m_1^i (g_1^i) q_1^i (e_0^i \cdot_F d_2(f_2)) \\
    &= -m_1^i (g_1^i) m_2^i (q_2^i (e_0^i \cdot_F f_2) ) .\\
\end{align*}
\endgroup
Thus, 
\begingroup\allowdisplaybreaks
\begin{align}\label{eq:case9}
    &g_2^j + \sum_{k \neq i }^t {g'_2}^{k,j} + m_1^i (g_1^i) q_2^j (e_0^i \cdot_F f_2), \ \textrm{and} \\
    &- g_1^i \cdot q_1^i (f_2) d_1 (e_0^i) + g_2^i + \sum_{k\neq i}^t {g'_2}^{k,i} + m_1^i(g_1^i) q_2^i (e_0^i \cdot_F f_2)
\end{align}
\endgroup
are both cycles, implying there exist $g_3^j \in G_3^j$ such that
$$m_3^j (g_3^j) = \begin{cases}
- g_1^i \cdot q_1^i (f_2) d_1 (e_0^i) + g_2^i + \sum_{k\neq i}^t {g'_2}^{k,i} + m_1^i (g_1^i)q_2^i (e_0^i \cdot_F f_2) & \textrm{if} \ i=j, \\
g_2^j + \sum_{k \neq i }^t {g'_2}^{k,j} + m_1^i (g_1^i)q_2^j (e_0^i \cdot_F f_2) & \textrm{otherwise}. \\
\end{cases}$$
Defining $$g_1^i \cdot_T f_2 := -m_1^i (g_1^i) e_0^i \cdot_F f_2 + \sum_{j=1}^t g_3^j,$$
one combines this with the first computation of this case to find:
\begingroup\allowdisplaybreaks
\begin{align*}
    \ell_3 (g_1^i \cdot_T f_2 ) &= \ell_3 ( - m_1^i (g_1^i) e_0^i \cdot_F f_2 + \sum_{j=1}^t g_3^j) \\
    &= - m_1^i (g_1^i) d_3 (e_0^i \cdot_F f_2 ) + \sum_{j=1}^t -m_1^i (g_1^i) q_2^j (e_0^i \cdot_F f_2 ) +\sum_{j=1}^t m_3^j( g_3^j) \\
    &= \ell_1 (g_1^i) f_2 - g_1^i \cdot_T \ell_2 (f_2). \\
\end{align*}
\endgroup

\end{proof}

As previously mentioned, in the case that each $G_\bullet^i$ has an additional DG-module structure, some of the products of Theorem \ref{thm:DGlength3} may be made more explicit.

\begin{prop}\label{prop:explicitProducts}
Adopt notation and hypotheses as in the statement of Theorem \ref{thm:DGlength3}, and assume that $G_\bullet^i$ admits the structure of a DG-module over each $G_\bullet^j$ for all $1 \leq i, j \leq t$. Then, the following products can be extended to a DG-algebra structure on $T_\bullet$:
\begingroup\allowdisplaybreaks
\begin{align*}
    (1) \quad F_1' \otimes G_1^i &\to F_2 \oplus \Big( \bigoplus_{j=1}^t G_2^j \Big) \\
    f_1 \cdot_T g_1^i &:= m_1^i(g_1^i) e_0^i \cdot_F f_1 + \sum_{j=1}^t g_1^i \cdot_{G^j} q_1^j (e_0^i \cdot_F f_1) , \\
    (2) \quad G_1^i \otimes G_1^j & \to  F_2 \oplus \Big( \bigoplus_{k=1}^t G_2^k \Big) \qquad i < j \\
    g_1^i \cdot_T g_1^j &:= m_1^i (g_1^i) m_1^j (g_1^j) e_0^i \cdot_F e_0^j + \sum_{k=1}^t g_2^k, \\
    \textrm{where} \ g_2^k &= \begin{cases}
        m_1^j (g_1^j) g_1^i \cdot_{G^i} q_1^i (e_0^i \cdot_F e_0^j) & \textrm{if} \ k=i \\
        m_1^i (g_1^i) g_1^j \cdot_{G^j} q_1^j (e_0^i \cdot_F e_0^j) & \textrm{if} \ k=j \\
        m_1^i (g_1^i) g_1^j \cdot_{G^k} q_1^k (e_0^i \cdot_F e_0^j) & \textrm{otherwise}, \\
    \end{cases} \\
    (3) \quad F_1' \otimes G_2^i &\to F_3 \oplus \Big( \bigoplus_{j=1}^t G_3^j \Big) \\
    f_1 \cdot_T g_2^i &:= -\sum_{j=1}^t g_2^i \cdot_{G^j} q_1^j (e_0^i \cdot_F f_1) ,  \\
    (4) \quad G_1^i \otimes G_2^j &\to F_3 \oplus \Big( \bigoplus_{k=1}^t G_3^k \Big) \\
    g_1^i \cdot_T g_2^j &:= \begin{cases}
        - \sum_{k \neq i} m_1^i (g_1^i) g_2^j \cdot_{G^k} q_1^k (e_0^i \cdot_F e_0^j) & \textrm{if} \ i<j \\
        - m_1^i (g_1^i) g_2^j \cdot_{G^j} q_1^j (e_0^i \cdot_F e_0^j) & \textrm{if} \ j<i ,\\
    \end{cases} \\
\end{align*}
\endgroup
\end{prop}

\begin{remark}
Observe that the assumption that $G_\bullet^i$ is a DG-module over each $G_\bullet^j$ is satisfied if $G_\bullet^i = G_\bullet^j$ for all $1 \leq i, \ j \leq t$.
\end{remark}

\begin{proof}
In order to show that the products in the statement of the Proposition are well-defined and may be extended to a product on all of $T_\bullet$, one only needs to verify the identities in the statement and proof of Theorem \ref{thm:DGlength3}. The verification is split into all $4$ cases:

\textbf{Case 1:} $F_1' \otimes G_1^i \to F_2 \oplus \Big( \bigoplus_{j=1}^t G_2^j \Big)$. One has:
\begingroup\allowdisplaybreaks
\begin{align*}
    m_2^j ( g_1^i \cdot_{G^j} q_1^j (e_0^i \cdot_F f_1) ) &= m_1^i (g_1^i) q_1^j (e_0^i \cdot_F f_1) - g_1^i \cdot \pi^j (d_2 (e_0 \cdot_F f_1) ) \\
    &= m_1^i (g_1^i) q_1^j (e_0^i \cdot_F f_1) + \begin{cases}
        d_1 (f_1) g_1^i & \textrm{if} \ i=j \\
        0 & \textrm{otherwise}.
    \end{cases} \\
\end{align*}
\endgroup
\textbf{Case 2:} $G_1^i \otimes G_1^j  \to  F_2 \oplus \Big( \bigoplus_{k=1}^t G_2^k \Big)$. One computes:
\begingroup\allowdisplaybreaks
\begin{align*}
    m_2^i (g_2^i) &= m_1^j (g_1^j ) m_1^i (g_1^i) q_1^i (e_0^i \cdot_F e_0^j) - m_1^j (g_1^j) g_1^i \cdot \pi^i (d_2 (e_0^i \cdot_F e_0^j)) \\
    &= m_1^j (g_1^j ) m_1^i (g_1^i) q_1^i (e_0^i \cdot_F e_0^j) + m_1^j (g_1^j) d_1 (e_0^j) g_1^i , \\
    m_2^j (g_2^j) &= m_1^i (g_1^i) m_1^j (g_1^j) q_1^j (e_0^i \cdot_F e_0^j) - m_1^i (g_1^i) g_1^j \cdot \pi^j (d_2 (e_0^i \cdot_F e_0^j)) \\
    &=m_1^i (g_1^i) m_1^j (g_1^j) q_1^j (e_0^i \cdot_F e_0^j) - m_1^i (g_1^i ) d_1 (e_0^i) g_1^j, \\
    m_2^k (g_2^k) &= m_1^i (g_1^i) m_1^j (g_1^j) q_1^k (e_0^i \cdot_F e_0^j) - m_1^i (g_1^i) g_1^j \cdot \pi^k (d_2 (e_0^i \cdot_F e_0^j)) \\
    &=m_1^i (g_1^i) m_1^j (g_1^j) q_1^j (e_0^i \cdot_F e_0^j) . \\
\end{align*}
\endgroup

\textbf{Case 3:} $ F_1' \otimes G_2^i \to F_3 \oplus \Big( \bigoplus_{j=1}^t G_3^j \Big)$. The identity \ref{eq:case6} must be verified. One computes:
\begingroup\allowdisplaybreaks
\begin{align*}
    m_3^j (-g_3^j) &= -m_2^i (g_2^i) \cdot_{G^j} q_1^j (e_0^i \cdot_F f_1) - g_2^i \cdot \pi^j (d_2 (e_0^i \cdot_F f_1) ) \\
    &= d_1 (f_1) g_2^i -  m_2^i (g_2^i) \cdot_{G^j} q_1^j (e_0^i \cdot_F f_1) \quad \textrm{if} \ i=j, \\
    &= -m_2^i (g_2^i) \cdot_{G^j} q_1^j (e_0^i \cdot_F f_1), \quad \textrm{otherwise.} \\
\end{align*}
\endgroup

\textbf{Case 4:} $ G_1^i \otimes G_2^j \to F_3 \oplus \Big( \bigoplus_{k=1}^t G_3^k \Big)$. The identities of \ref{eq:case8} must be verified. One computes:
\begingroup\allowdisplaybreaks
\begin{align*}
    m_3^k (m_1^i (g_1^i) g_2^j \cdot_{G^k} q_1^k (e_0^i \cdot_F e_0^j)) &= m_1^i (g_1^i) m_2^j (g_2^j) \cdot_{G^k} q_1^k (e_0^i \cdot_F e_0^j) \\
    &+ \begin{cases}
        m_1^i (g_1^i) d_1 (e_0^i) g_2^j & \textrm{if} \ k=j \\
        0 & \textrm{otherwise}  \\
    \end{cases} \quad \textrm{if} \ i<j \\
    m_3^j (m_1^i (g_1^i) g_2^j \cdot_{G^j} q_1^j (e_0^i \cdot_F e_0^j) ) &= m_1^i (g_1^i) m_2^j(g_2^j) \cdot_{G^j} q_1^j (e_0^i \cdot_F e_0^j ) \\
    & +m_1^i (g_1^i) d_1 (e_0^j) g_2^j \\
    &= g_1^i \cdot_T m_2^j (g_2^j) + m_1^i (g_1^i) d_1 (e_0^j) g_2^j\quad \textrm{if} \ j<i. \\
\end{align*}
\endgroup
This completes the proof.
\end{proof}

 Recall that the notation $\overline{\cdot}$ denotes the functor $\cdot \otimes_R k$. The following corollaries are immediate consequences of the statement and proof of Theorem \ref{thm:DGlength3}.

\begin{cor}\label{cor:nontrivialMults}
Adopt notation and hypotheses as in Setup \ref{set:trimmingcxSetup}. Assume that the complexes $F_\bullet$ and $G_\bullet^i$ (for $1 \leq i \leq t$) are minimal. Then the only possible nontrivial products in the algebra $\overline{T_\bullet}$ are
$$\overline{F_1'} \cdot_T \overline{F_1'}, \quad \overline{F_1'} \cdot_T \overline{F_2}, \quad \overline{F_1'} \cdot_T \overline{G_1^i} \quad \overline{G_1^i} \cdot_T \overline{F_2}, \quad \textrm{and} \quad \overline{F_1'} \cdot_T \overline{G_2^i} \quad (1 \leq i \leq t).$$
Moreover, the map
\begingroup\allowdisplaybreaks
\begin{align*}
    \overline{T_\bullet} &\to \overline{F_\bullet}, \quad \textrm{defined by} \\
    \overline{f_i} &\mapsto \overline{f_i}, \qquad (f_i \in F_i), \\
    \overline{g_i^j} &\mapsto 0, \qquad (g_i^j \in G_i^j),
\end{align*}
\endgroup
is a homomorphism of $k$-algebras.
\end{cor}

\begin{proof}
It is clear that the products $(3)$ and $(7)$ in the statement of Theorem \ref{thm:DGlength3} are trivial after tensoring with $k$. For the products $(4)$ and $(8)$, observe that each $g_2^i$ or $g_3^i$, respectively, can be chosen up to a cycle to be a constant multiple of $m_1^i (g_1^i)$ for some $1 \leq i \leq t$. The latter statement follows immediately from the form of the products given in Theorem \ref{thm:DGlength3}.
\end{proof}

\section{Consequences for Tor Algebra Structures}\label{sec:ToralgCons}

In this section, we take advantage of Corollary \ref{cor:nontrivialMults} and study how the process of trimming an ideal affects the Tor-algebra class. As it turns out, if the multiplication between certain homological degrees has coefficients appearing in sufficiently high powers of the maximal ideal, then the Tor-algebra class will be preserved. We also give explicit examples showing that if this assumption is not satisfied, then it is possible to obtain new nontrivial multiplication in the associated trimming complex. We begin with the Tor-algebra classification provided by Avramov, Kustin, and Miller in \cite{avramov1988poincare}.

\begin{theorem}[\cite{avramov1988poincare}, Theorem $2.1$]\label{thm:Toralgs}
Let $(R , \m , k)$ be a regular local ring and $I \subset R$ any ideal with $\pd_R (R/I) = 3$. There are nonnegative integers $p$, $q$, and $r$ and bases $\{ \overline{f_1^i} \}$, $\{ \overline{f_2^i} \}$, and $\{ \overline{f_3^i} \}$ for $\tor_1^R (R/I , k)$, $\tor_2^R (R/I , k)$, and $\tor_3^R (R/I , k)$, respectively, such that the multiplication in $\tor_+^R (R/I , k)$ is given by one of the following:
\begingroup\allowdisplaybreaks
\begin{align*}
    \textrm{CI}: \qquad &\overline{f_2^1} = \overline{f_1^2} \overline{f_1^3}, \ \overline{f_2^2} = \overline{f_1^3} \overline{f_1^1}, \ \overline{f_2^3} = \overline{f^1_1} \overline{f_1^2} \\
    &\overline{f_1^i} \overline{f_2^j} = \delta_{ij} \overline{f_3^1} \ \textrm{for} \ 1 \leq i, \ j \leq 3 \\
    \textrm{TE}: \qquad &\overline{f_2^1} = \overline{f_1^2} \overline{f_1^3}, \ \overline{f_2^2} = \overline{f_1^3} \overline{f_1^1}, \ \overline{f_2^3} = \overline{f^1_1} \overline{f_1^2} \\
    \textrm{B}: \qquad & \overline{f_1^1} \overline{f_1^2} = \overline{f_2^3}, \ \overline{f^1_1} \overline{f_2^1} = \overline{f_3^1} , \overline{f_1^2} \overline{f_2^2} = \overline{f_3^1} \\
    \textrm{G} (r): \qquad & \overline{f_1^i} \overline{f_2^i} = \overline{f_3^1}, \ 1 \leq i \leq r \\
    \textrm{H} (p,q): \qquad & \overline{f_1^{p+1}} \overline{f_1^i} = \overline{f_2^i}, \ 1 \leq i \leq p \\
    & \overline{f_1^{p+1}} \overline{f_2^{p+i}} = \overline{f_3^i}, \ 1 \leq i \leq q \\
\end{align*}
\endgroup
\end{theorem}

\begin{remark}\label{rk:theTuples}
In terms of the tuples presented in the Realizability Question \ref{question:realizability}, the classes $G(r)$ and $H(p,q)$ have:
\begingroup\allowdisplaybreaks
\begin{align*}
    \textrm{G} (r) : \qquad & (m,n,0,1,r) \\
    \textrm{H} (p,q) : \qquad & (m,n,p,q,q). \\
\end{align*}
\endgroup
For this reason, a tuple coming from a class $G(r)$ ring will often be shortened to $(m,n,r)$, and a tuple coming from a class $H(p,q)$ ring will be shortened to $(m,n,p,q)$. These tuples will be referred to as \emph{the associated tuple}. 
\end{remark}

The following definition is introduced out of convenience for stating the results appearing later in this section.

\begin{definition}
A Tor-algebra is in \emph{standard form} if the basis elements have been chosen such that the multiplication is given by one of the possibilities of Theorem \ref{thm:Toralgs}. A DG-algebra free resolution $F_\bullet$ is in standard form if the multiplication descends to a Tor-algebra in standard form after applying $- \otimes_R k$. 
\end{definition}

\begin{example}[{\cite[Example 2.1.3]{avramov1998infinite}}]\label{example:Pfaffproduct}
Let $R = k[x_{ij} \mid 1 \leq i< j \leq n]$ and $X = (x_{ij} )$ be a generic $n \times n$ skew symmetric matrix, with $n$ odd. Given an indexing set $L = (i_1 < \cdots < i_k)$, let 
$$\textrm{Pf}_L (X) := \textrm{pfaffian of} \ X \ \textrm{with rows and columns from} \ L \ \textrm{removed.}$$
Define
$$d_1 := (\textrm{Pf}_1 (X), - \textrm{Pf}_2 (X) , \dots , (-1)^{i+1} \textrm{Pf}_i (X) , \dots , \textrm{Pf}_n (X)),$$
and consider the complex
$$F_\bullet: \qquad \xymatrix{0 \ar[r] & R  \ar[r]^-{d_1^*} & R^n \ar[r]^-{X} & R^n \ar[r]^-{d_1} & R \ar[r] & 0}.$$
Then $F_\bullet$ admits the structure of an associative DG-algebra with the following products:
$$f_1^i \cdot_F f_1^j = \sum_{k=1}^n (-1)^{i+j+k} \textrm{Pf}_{ijk} (X) f_2^k,$$
$$f_1^i \cdot_F f_2^j =\delta_{ij} f_3^1,$$
where $\{ f_1^i \}_{1 \leq i \leq n}$, $\{ f_2^i \}_{1 \leq i \leq n}$, and $\{ f_3^1 \}$ are standard basis elements for $F_1$, $F_2$, and $F_3$, respectively. If $n \geq 5$, then $F_\bullet$ is in standard form since it descends to a Tor-algebra of class $G(n)$ in standard form.
\end{example}

\begin{notation}
If $A = A_3 \oplus A_2 \oplus A_1 \oplus k$ is a finite dimensional graded-commutative $k$-algebra, then
$$A_2^\perp := \{ a \in A_1 \mid a \cdot A_2  = 0 \}.$$
\end{notation}

The following Lemma is a coordinate free characterization for length $3$ $k$-algebras realizing certain types of algebra classes. If $V$ is vector space with subspace $W \subset V$, define $\codim_V W := \dim V - \dim W$. 

\begin{lemma}[\cite{avramov1988poincare}, Lemma $2.3$]\label{lem:coordfreeClass}
Suppose $A = A_3 \oplus A_2 \oplus A_1 \oplus k$ is a finite dimensional graded-commutative $k$-algebra with $A_1^2 = 0$. Then:
\begin{enumerate}[(a)]
    \item $A$ has algebra structure of the form $H(0,0)$ if and only if $A_1 \cdot A_2 = 0$. 
    \item $A$ has algebra structure of the form $H(0,q)$ for some $q \geq 1$ if and only if $\codim_{A_1} A_2^\perp = 1$ and $\dim A_1 \cdot A_2 = q$.
    \item $A$ has algebra structure of the form $G(r)$ for some $r \geq 2$ if and only if $\dim A_1 \cdot A_2 = 1$ and $\codim_{A_1} A_2^\perp = r$.
\end{enumerate}
\end{lemma}

\begin{notation}
Let $I = (\phi_1 , \dots  , \phi_n) \subseteq R$ be an $\m$-primary ideal and $F_\bullet$ a DG-algebra free resolution of $R/I$ in standard form. Given an indexing set $\sigma = (1 \leq \sigma_1 < \cdots < \sigma_t \leq n )$, define
$$\tm_\sigma (I) := (\phi_i \mid i \notin \sigma) + \m (\phi_j \mid j \in \sigma).$$
The transformation $I \mapsto \tm_\sigma (I)$ will be referred as \emph{trimming} the ideal $I$.
\end{notation}

The following setup will be in effect for the remainder of this section.

\begin{setup}\label{set:trimmingSetup}
Let $(R , \m , k)$ be a regular local ring of dimension $3$. Let $I = (\phi_1 , \dotsc , \phi_n) \subseteq R$ be an ideal defining a quotient of projective dimension $3$ and $\sigma = (1 \leq \sigma_1 < \cdots < \sigma_t \leq n)$ be an indexing set. Let $(F_\bullet, d_\bullet)$ and $(K_\bullet, m_\bullet)$ be minimal DG-algebra free resolutions and $R/I$ and $k$, respectively. By Theorem \ref{itres}, a free resolution of $R/ \tm_\sigma (I)$ may be obtained as the mapping cone of a morphism of complexes of the form:
\begin{equation}\label{eq:trimcx}
    \xymatrix{0  \ar[rr] && F_3 \ar[d]^-{Q_2} \ar[rr]^-{d_3} && F_2 \ar[d]^-{Q_1} \ar[rrr]^-{d_2'} &&& F_1' \ar[d]^-{d_1'} \\
\bigoplus_{i=1}^t K_3 \ar[rr]^-{\bigoplus_{i=1}^t m_3} && \bigoplus_{i=1}^t K_2 \ar[rr]^-{\bigoplus_{i=1}^t m_2} && \bigoplus_{i=1}^t K_1 \ar[rrr]^-{-\sum_{i=1}^t m_1(-) d_1(e_0^{\sigma_i})} &&& R, \\}
\end{equation}
where
$$F_1' := \bigoplus_{j \notin \sigma} Re_0^j \quad \textrm{and} \quad d_2' : F_2 \xrightarrow{d_2} F_1 \xrightarrow{\textrm{proj}} F_1',$$
and $d_1' : F_1' \to R$ is the restriction of $d_1$ to $F_1'$. Let $T_\bullet$ denote the mapping cone of \ref{eq:trimcx}; recall that the differentials of $T_\bullet$ are written explicitly in Remark \ref{rk:theDiffs}. 
\end{setup}

The following Lemma says that in the context of Setup \ref{set:trimmingSetup}, the possible nontrivial multiplications in the homology algebra are even more restricted than that of Corollary \ref{cor:nontrivialMults}.

\begin{lemma}\label{lem:trivialproducts}
Adopt notation and hypotheses as in Setup \ref{set:trimmingSetup} and assume that $I \subseteq \m^2$. Then,
$$\overline{F_1} \cdot_F \overline{F_1} \subseteq \ker (Q_1 \otimes k)$$
$$\overline{F_1} \cdot_F \overline{F_2} \subseteq \ker (Q_2 \otimes k)$$
In particular, the only possible nontrivial products in the algebra $\overline{T_\bullet}$ are given by
$$\overline{F_1'} \cdot_T \overline{F_1'} \quad \textrm{and} \quad \overline{F_1'} \cdot_T \overline{F_2}.$$
\end{lemma}

\begin{proof}
One has:
\begingroup\allowdisplaybreaks
\begin{align*}
    d_2 (F_1 \cdot_F F_1) &\subseteq \m^2 F_1 \\
    \implies  m_1^{\sigma_i} \circ q_1^{\sigma_i} (F_1 \cdot_F F_1) &= {d_0'}^{\sigma_i} (F_1 \cdot_F F_1) \subseteq \m^2 \\
    \implies q_1^{\sigma_i} (F_1 \cdot_F F_1) &\subseteq \m K_1 \ \textrm{for all} \ 1 \leq i \leq t. \quad (*)\\
\end{align*}
\endgroup
Likewise,
\begingroup\allowdisplaybreaks
\begin{align*}
    m_2^{\sigma_i} (q^{\sigma_i}_2 (F_1 \cdot_F F_2)) &= q_1^{\sigma_i} ( d_3 (F_1 \cdot_F F_2)) \\
    &\subseteq \m \cdot q_1^{\sigma_i} ( F_1 \cdot_F F_1 ) \\
    & \subseteq \m^2 K_1 \qquad \textrm{(by (*) above)}\\
    \implies q^{\sigma_i}_2 (F_1 \cdot_F F_2) &\subseteq \m K_2 \ \textrm{for all} \ 1 \leq i \leq t. \\
\end{align*}
\endgroup
 The latter claim about the triviality of the products $\overline{F_1'} \cdot \overline{K_1}$, $\overline{K_1} \cdot_T \overline{F_2}$, and $\overline{F_1'} \cdot_T \overline{K_2}$ follows immediately from the definition of the products given in Theorem \ref{thm:DGlength3} (and the more explicit forms given in Proposition \ref{prop:explicitProducts}) combined with the above properties. 
\end{proof}


\begin{notation}
Let $a$ and $b$ be positive integers with $a<b$. The notation $[a]$ will denote the set $\{1 , 2 , \dots , a-1 , a \}$ and the notation $[a,b]$ will denote the set $\{a , a+1 , \dots , b-1 , b \}$. 
\end{notation}

The following lemma makes precise the previously mentioned fact that, under mild hypotheses, the Tor-algebra class $G(r)$ is preserved by trimming.

\begin{lemma}\label{lem:classGr}
Adopt notation and hypotheses as in Setup \ref{set:trimmingSetup}. Assume $R/I$ defines a ring of class $G (r)$ and $F_\bullet$ is in standard form with $F_1 \cdot_F F_1 \subseteq \m^2 F_2$. Then for all indexing sets $\sigma$, the ideal $\tm_\sigma (I)$ defines either a Golod ring or a ring of Tor algebra class $G(r')$, for some $r'$ that satisfies the inequalities $r- |\sigma \cap [r]| - \rank_k (Q_1 \otimes k) \leq r' \leq r - |\sigma \cap [r]|$. 
\end{lemma}

\begin{proof}
By Lemma \ref{lem:trivialproducts}, $f_3^1$ may be chosen as part of a basis for $\ker (Q_2 \otimes k)$; the parameter $r'$ arises from counting the rank of the induced map
$$\ker (Q_1 \otimes k) \to \hom_k (\overline{F'_1} , \overline{F_3}).$$
By definition of the product on $T_\bullet$, the assumption $F_1 \cdot_F F_1 \subseteq \m^2 F_2$ implies that 
$$\overline{f_1} \cdot_T \overline{f_1'} = \overline{f_1} \cdot_F \overline{f_1'} = 0 \ \textrm{for all} \ f_1, f_1' \in F_1'.$$
Thus, the induced map $\overline{F_2} \to \hom_k (\overline{F'_1} , \overline{F_3})$ has rank $r-|\sigma \cap [r]|$. But this map may be written as the composition
$$\ker (Q_1 \otimes k) \hookrightarrow \overline{F_2} \to \hom_k (\overline{F'_1} , \overline{F_3}),$$
whence one finds that $r-|\sigma \cap [r]| - \rank_k (Q_1 \otimes k) \leq r' \leq r-|\sigma \cap [r]|$. 
\end{proof}

\begin{cor}\label{cor:quadraticGr}
Adopt notation and hypotheses as in the statement of Lemma \ref{lem:classGr}. If $d_2 (F_2) \subseteq \m^2 F_1$, then $\tm_\sigma (I)$ defines a ring of Tor-algebra class $G(r-|\sigma \cap [r]| )$. 
\end{cor}

\begin{proof}
The assumption $d_2 (F_2) \subseteq \m^2 F_1$ implies that $Q_1 \otimes k = 0$, so the inequality on $r'$ as in Lemma \ref{lem:classGr} implies that $r' = r - |\sigma \cap [r]|$.
\end{proof}
The following example shows that the assumption $F_1 \cdot_F F_1 \subseteq \m^2 F_2$ in Lemma \ref{lem:classGr} is necessary.

\begin{example}
Let $R = k[x_1 , x_2 , x_3]$ and
$$X := \begin{pmatrix}
       0&0&0&{x}_{1}&{x}_{2}\\
       0&0&{x}_{1}&{x}_{2}&{x}_{3}\\
       0&-{x}_{1}&0&{x}_{3}&0\\
       -{x}_{1}&-{x}_{2}&-{x}_{3}&0&0\\
       -{x}_{2}&-{x}_{3}&0&0&0\end{pmatrix}.$$
Let $I = \textrm{Pf} (X) = ({x}_{3}^{2},\,-{x}_{2}{x}_{3},\,{x}_{2}^{2}-{x}_{1}{x}_{3},\,-{x}_{1}{x}_{2},\,{x}_{1}^{2})
$, the ideal of $4 \times 4$ pfaffians of $X$. The ring $R/I$ has Tor-algebra class $G(5)$ and the multiplication satisfies $F_1 \cdot F_1 \subseteq \m F_2$ and $F_1 \cdot F_1 \not\subseteq \m^2 F_2$. It may be shown using Macaulay2 \cite{M2} that $\tm_1 (I)$ defines a ring of Tor-algebra class $B$ and $\tm_2 (I)$ defines a ring of Tor-algebra class $H(3,2)$. Both of these Tor-algebras have nontrivial multiplication of elements in homological degree $1$, which shows that the multiplication on $\overline{T_\bullet}$ cannot possibly agree with the multiplication on $\overline{F_\bullet}$.
\end{example}

It turns out that trimming also tends to preserve the Tor-algebra class $H(p,q)$:

\begin{lemma}\label{lem:classHpq}
Adopt notation and hypotheses as in Setup \ref{set:trimmingSetup} and assume that $F_\bullet$ is in standard form of class $H(p,q)$ with the property that
$$f_1^i \cdot_F f_1^j \in \m^2 F_2 \ \textrm{for all} \ i,j \neq p+1.$$
Then,
\begin{enumerate}
        \item if $p+1 \in \sigma$, $\tm_\sigma (I)$ defines a Golod ring, and
        \item if $p+1 \notin \sigma$, then $\tm_\sigma (I)$ defines either a Golod ring or a ring of class $H(p- |\sigma \cap [p]|,q')$, for some $q'$ satisfying the inequalities $ q- \rank_k (Q_1 \otimes k)  \leq q' \leq q$.
    \end{enumerate}
\end{lemma}

\begin{proof}
In an identical manner to the proof of Lemma \ref{lem:classGr}, the assumption $f_1^i \cdot_F f_1^j \in \m^2 F_2$ implies
$$\overline{f_1^i} \cdot_T \overline{f_1^j} = \overline{f_1^i} \cdot_F \overline{f_1^j} = 0 \quad \textrm{for all} \ i,j \neq p+1.$$
\textbf{Case 1:} $p+1 \in \sigma$. Since $f_1^{p+1} \notin F_1'$, it follows that $\overline{F_1'} \cdot_T \overline{F_1'} = 0$. For identical reasons, $\overline{F_1'} \cdot_T \overline{F_2} = 0$. Thus, the multiplication in the Tor-algebra is trivial, so $R/I$ is a Golod ring.

\textbf{Case 2:} $p+1 \notin \sigma$. By Lemma \ref{lem:trivialproducts}, $\overline{f_2^1} , \dots , \overline{f_2^p}$ may be chosen as part of a basis of $\ker (Q_1 \otimes k)$. This immediately implies that the only nontrivial products in $\overline{T_\bullet}$ are of the form
$$\overline{f_1^i} \cdot_T \overline{f_1^{p+1}} = \overline{f_2^i} \ \textrm{for} \ 1 \leq i \leq p, \ i \notin \sigma.$$
Thus, $\dim_k \overline{F_1'} \cdot_T \overline{F_1'} = p-|\sigma \cap [p]|$. Likewise, $\overline{f_3^{p+1}} , \dots , \overline{f_3^{p+q}}$ may be chosen as part of a basis for $\ker (Q_2 \otimes k)$. Moreover, the rank of the induced map
$$\ker (Q_1 \otimes k) \to \hom_k (\overline{F_1'} , \overline{F_3})$$
is at least $q - \rank_k (Q_1 \otimes k)$, and it is evidently at most $q$. 
\end{proof}

\begin{cor}\label{cor:quadraticHpq}
Adopt notation and hypotheses as in the statement of Lemma \ref{lem:classHpq}. If $d_2 (F_2) \subseteq \m^2 F_2$, then $\tm_\sigma (I)$ is either Golod or defines a ring of Tor-algebra class $H(p-|\sigma \cap [ p]| , q)$. 
\end{cor}
Again, the assumption that $f_1^i \cdot_F f_1^j \in \m^2 F_2$ for $i,j \neq p+1$ in Lemma \ref{lem:classHpq} is necessary, as the following example shows.

\begin{example}
Let $R = k[x_1 , x_2 , x_3]$ and $I = ({x}_{2}^{2}-{x}_{1}{x}_{3},\,-{x}_{1}{x}_{2},\,{x}_{1}^{2}, \ {x}_{3}^{2} )$. The ring $R/I$ has Tor-algebra class $H(3,2)$, and just by counting degrees one knows the multiplication on the minimal free resolution $F_\bullet$ of $R/I$ satisfies $f_1^i \cdot_F f_1^j \in \m F_2$, $f_1^i \cdot_F f_2^j \notin \m^2 F_2$, where $i, \ j \neq 5$. However, it can be shown using Macaulay2 \cite{M2} that $\tm_2 (I)$ has Tor-algebra class TE.
\end{example}

\section{Examples}\label{sec:examples}

In this section, we employ the theory developed in Section \ref{sec:ToralgCons} for the construction of explicit examples of rings realizing Tor-algebra classes $G(r)$ and $H(p,q)$ for a given set of parameters $(m,n,p,q,r)$ (as in the introductory discussion and Question \ref{question:realizability}). These examples are constructed in an arbitrary regular local ring $(R , \m , k)$; in particular, we will construct explicit novel examples of ideals defining rings of Tor-algebra class $G(r)$ (and arbitrarily large type). Combining the process of trimming with that of \emph{linkage}, which also tends to preserve Tor-algebra class (see \cite{christensen2020linkage} for more on this), one can link the examples constructed in this section to obtain an even larger class of tuples. In the case of class $H(p,q)$ examples, this will result in another ideal defining a ring of Tor-algebra class $H$ with parameter adjusted accordingly as in the statement of \cite[Theorem 1.1]{christensen2020linkage}.

\begin{setup}\label{set:examplesSet}
Let $(R , \m ,k)$ denote a regular local ring of dimension $3$ (or a standard graded $3$-variable polynomial ring over a field). Let $\m = (x_1 , x_2 , x_3)$, where $x_1$, $x_2$, $x_3$ is a regular sequence. Let $K_\bullet$ denote the Koszul complex with $K_1 := Re_1 \oplus Re_2 \oplus Re_3$ induced by the map sending $e_i \mapsto x_i$. 
\end{setup}

The matrices appearing in the following two definitions were inspired by matrices constructed in \cite{christensen2019trimming} and further generalized in \cite{vandebogert2019structure}. Here, we extend this definition to arbitrary local rings for the construction of our examples.

\begin{definition}
Adopt notation and hypotheses as in Setup \ref{set:examplesSet}. Let $U_m^{j}$ (for $j \leq m$) denote the $m\times m$ matrix with entries from $R$ defined by:
$$(U^{j}_m)_{i,m-i} = x_1^2, \quad (U^{j}_m)_{i,m-i+1} = x_3^2, \quad (U^{j}_m)_{i,m-i+2} = x_2^2 \ \textrm{for} \  i \leq m-j$$
$$(U^{j}_m)_{i,m-i} = x_1, \quad (U^{j}_m)_{i,m-i+1} = x_3, \quad (U^{j}_m)_{i,m-i+2} = x_2 \ \textrm{for} \  i >m-j$$
and all other entries are defined to be $0$.
\end{definition}

To see the pattern, observe that:
$$U_2^1 = \begin{pmatrix}
       x_1^{2}&x_3^{2}\\
       x_3&x_2\end{pmatrix}, \ U_3^1 = \begin{pmatrix}
       0&x_1^{2}&x_3^{2}\\
       x_1^{2}&x_3^{2}&x_2^{2}\\
       x_3&x_2&0\end{pmatrix}, \ U_3^2 = \begin{pmatrix}
       0&x_1^{2}&x_3^{2}\\
       x_1&x_3&x_2\\
       x_3&x_2&0\end{pmatrix}$$
\begin{definition}\label{def:Vmatrices}
Define $V_m^{j}$ (for $j< m$) to be the $(2m+1)\times (2m+1)$ skew symmetric matrix
$$V^{j}_m := \begin{pmatrix}
O & O_{x_1^2} & (U_m^{j})^T \\
-(O_{x_1^2})^T & 0 & ^{x_2^2}O \\
-U_m^{j} & -(^{x_2^2}O)^T & O \\
\end{pmatrix}.$$
If $j=m$, then $V_m^m$ is the skew symmetric matrix
$$V^{m}_m := \begin{pmatrix}
O & O_{x_1^2} & (U_m^{m})^T \\
-(O_{x_1^2})^T & 0 & ^{x_2}O \\
-U_m^{m} & -(^{x_2}O)^T & O \\
\end{pmatrix}.$$
Lastly, if $j = m+1$, then $V^{m+1}_m$ is the skew symmetric matrix
$$V^{m+1}_m := \begin{pmatrix}
O & O_{x_1} & (U_m^{m})^T \\
-(O_{x_1})^T & 0 & ^{x_2}O \\
-U_m^{m} & -(^{x_2}O)^T & O \\
\end{pmatrix}.$$
In the above, $^fO$, $O^f$, $O_f$, and $_fO$ denote the appropriately sized matrices with $f$ in the top left, top right, bottom right, and bottom left corners, respectively, and $0$ in all other entries.
\end{definition}

\begin{definition}\label{def:pfaffIdeal}
Let $m \geq 2$ be an integer. Define the ideal $I_m^j$ (for $0 \leq j \leq m+1$) by
$$I_m^j := \textrm{Pf} (V_m^j),$$
where $\textrm{Pf} (V_m^j)$ denotes the ideal of $2m \times 2m$ pfaffians of $V_m^j$.
\end{definition}

\begin{setup}\label{set:pfaffsSetup}
Adopt notation and hypotheses as in Setup \ref{set:examplesSet}. Define
$$d_1 := (\textrm{Pf}_1 (V_m^j), - \textrm{Pf}_2 (V_m^j) , \dots , (-1)^{i+1} \textrm{Pf}_i (V_m^j) , \dots , \textrm{Pf}_{2m+1} (V_m^j)),$$
(for $m \geq 2$ and $0 \leq j \leq m+1$) and consider the complex
$$F_\bullet: \qquad \xymatrix{0 \ar[r] & R  \ar[r]^-{d_1^*} & R^n \ar[r]^-{V_m^j} & R^n \ar[r]^-{d_1} & R \ar[r] & 0}.$$
Recall that $F_\bullet$ is a minimal free resolution of $R / I_m^j$ in standard form of class $G(2m+1)$ with product as in Example \ref{example:Pfaffproduct}. 
\end{setup}

\begin{prop}\label{prop:expQmapGr}
Adopt notation and hypotheses as in Setup \ref{set:pfaffsSetup}. Define $q_1^i : F_2 \to K_1$ by sending:
\begingroup\allowdisplaybreaks
\begin{align*}
    f_2^{2m+3-i} & \mapsto \begin{cases}
        e_2  & \textrm{if} \ 1< i \leq j+1 \leq m+1 \\
        x_2 e_2 & \textrm{if} \ j+1 < i \leq m+1 \\ 
        -x_2 e_2 & \textrm{if} \ m+1 < i \leq 2m+1-j \\
        - e_2 & \textrm{if} \ 2m+1-j < i \leq 2m+1, \\
    \end{cases} \\
    f_2^{2m+2-i} & \mapsto \begin{cases}
        e_3 & \textrm{if} \ 1 \leq i \leq j, \ i<m+1 \\
        x_3 e_3 & \textrm{if} \ j < i < m+1 \\
        0  & \textrm{if} \ i=m+1 \\
        -x_3 e_3 & \textrm{if} \ m+1 < i \leq 2m+1-j \\
        - e_3 & \textrm{if} \ 2m+1-j < i \leq 2m+1, \\
    \end{cases} \\
    f_2^{2m+1-i} & \mapsto \begin{cases}
        e_1 & \textrm{if} \ 1 \leq i \leq j-1 \\
        x_1 e_1 & \textrm{if} \ j-1 < i < m+1 \\ 
        -x_1 e_1 & \textrm{if} \ i=m+1, \ j< m+1 \\
        -e_1 & \textrm{if} \ i=m+1, \ j=m+1 \\
        -x_1 e_1 & \textrm{if} \ m+1 < i \leq 2m+1-j \\
        - e_1 & \textrm{if} \ 2m+1-j < i < 2m+1, \\
    \end{cases} \\
\end{align*}
\endgroup
and all other basis elements are sent to $0$. Then the following diagram commutes:
$$\xymatrix{F_2 \ar[dr]^-{{d_0'}^i} \ar[d]_-{q_1^i} & \\
K_1 \ar[r]^-{m_1} & \m. \\}$$
\end{prop}

\begin{cor}\label{cor:GrExamples}
Adopt notation and hypotheses as in Setup \ref{set:pfaffsSetup}.  Let $t \geq 1$ be any integer and
\begin{enumerate}[(a)]
    \item $m \geq 3$, or
    \item $m=2$ and $j =0$.
\end{enumerate}
Then:
\begin{enumerate}
    \item The ideal $\tm_i (I_m^j)$ defines either a Golod ring or a ring of Tor algebra class $G(r)$, where $r = 2m+1 - t - \rank_k (q_1^i \otimes k)$.
    \item If $t \leq 2m+1 - j$, then the ideal $\tm_{[t]} (I_m^j)$ defines either a Golod ring or a ring of Tor-algebra class $G(r)$, where $r= 2m+1-t - \min \{ 1+t , j \}$. 
    \item More generally, let $\sigma = (1 \leq \sigma_1 < \cdots < \sigma_t \leq 2m+1)$ denote an indexing of length $t$. Then the ideal $\tm_\sigma (I_m^j)$ defines either a Golod ring or a ring of Tor-algebra class $G(r)$, where
    $$r = 2m+1 -t - \rank_k (Q_1 \otimes k) + |\sigma \cap \{ j \mid \overline{Q_1 (f_2^j)} \neq 0 \}|.$$
\end{enumerate} 
\end{cor}

\begin{proof}
The assumptions $(a)$ and $(b)$ ensure that $F_1 \cdot F_1 \subseteq \m^2 F_2$, so that by Lemma \ref{lem:classGr}, $\tm_\sigma (I_m^j)$ defines either a Golod ring or a ring of class $G(r')$, where $2m - \rank_k (q_1^i \otimes k) \leq r' \leq 2m$. By construction, $\ker (Q_1 \otimes k)$ for each $i$ has basis given by 
$$\{ \overline{f_2^j} \in \overline{F_2} \mid \overline{Q_1 (f_2^j)} = 0 \}.$$
\textbf{Case 1:} By the above, one immediately has that $\rank_k \big( \ker (Q_1 \otimes k) \to \hom_k (\overline{F_1'}, \overline{F_3}) \big) = 2m- \rank_k (q_1^i \otimes k)$. This is because (by Proposition \ref{prop:expQmapGr})
$$i \notin \{ j \mid  \overline{Q_1 (f_2^j)} \neq 0 \}.$$
\textbf{Case 2:} Using Proposition \ref{prop:expQmapGr}, one finds that $\rank_k (Q_1 \otimes k) = \min \{ 1+r, j\}$. Moreover, since $t <2m+1-j$,
$$ [t] \cap \{  j \mid  \overline{Q_1 (f_2^j)} \neq 0 \} = \varnothing,$$
whence $\rank_k \big( \ker (Q_1 \otimes k) \to \hom_k (\overline{F_1'}, \overline{F_3}) \big) \leq 2m+1-t - \min \{ 1+t , j \}$. By Lemma \ref{lem:classGr}, one has equality. \\
\textbf{Case 3:} Define $S := \sigma \cap \{ j \mid \overline{Q_1 (f_2^j)} \neq 0 \}$. If $j \in S$, then in $T_\bullet$, the direct summand $Rf_1^j$ has been omitted from $F_1$. Thus, removal of the direct summand generated by $f_2^j$ has no effect on the rank of the induced map
$$\delta: \ker (Q_1 \otimes k) \to \hom_k (\overline{F_1'} , \overline{F_3}).$$
By inclusion-exclusion, this implies that $\delta$ has rank
$$2m+1-t - \rank_k (Q_1 \otimes k) + |S|.$$
\end{proof}

\begin{remark}
Let $S := \sigma \cap \{ j \mid \overline{Q_1 (f_2^j)} \neq 0 \}$. Then, in terms of the associated tuple (see Remark \ref{rk:theTuples}), the transformation $I_m^j \mapsto \tm_\sigma (I_m^j)$ transforms the tuple $(2m+1 , 1 , 2m+1)$ as so:
\begingroup\allowdisplaybreaks
\begin{align*}
    I_m^j &\mapsto \tm_\sigma (I_m^j) \\
    & (2m+1,1,2m+1) \\
    \mapsto& (2m+2t+1 - \rank_k (Q_1 \otimes k) , 1+t , 2m+ 1 - t - \rank_k(Q_1 \otimes k) + |S|) \\
\end{align*}
\endgroup
\end{remark}

Corollary \ref{cor:GrExamples} immediately allows us to fill in a large class of tuples:

\begin{cor}\label{cor:classGrealizable}
Adopt notation and hypotheses as in Setup \ref{set:trimmingSetup}. Let $(m,n,r)$ be a tuple of positive integers satisfying either:
\begin{enumerate}
    \item $r =m - 3(n-1) \geq 2$, $n \geq 2$, and $m \geq 2n+3 \geq 7$,
    \item $r =m- 3(n-1) + 2 \geq 2$, $n \geq 3$, and $m \geq 2n+1 \geq 7$, 
    \item $r = m- 3(n-2) \geq 2$, $n \geq 4$, and $m \geq 2n+1 \geq 9$.
\end{enumerate}
Then there exists an ideal $J$ defining a ring of Tor-algebra class $G(r)$ realizing this tuple.
\end{cor}

\begin{proof}
Observe that since the matrices of Definition \ref{def:Vmatrices} have odd sizes, the proof necessarily splits up into cases that depend on the parity of the parameters appearing in the tuple $(m,n,r)$. \\ 
\textbf{Case 1(a):} $n+r$ is even. Write $n+r = 2k+2$ for some integer $k \geq 2$. Consider the ideal $\tm_{[n-1]} (I_k^0)$. By Corollary \ref{cor:GrExamples}, this has the effect of transforming the associated tuple in the following way:
\begingroup\allowdisplaybreaks
\begin{align*}
    (2k+1, 1 , 2k+1) &\mapsto (2k+1+2(n-1),1 + n-1 , 2k+1 - (n-1)) \\
    &= (n+r+2n-3, n ,n+r-1-n+1) \\
    &= (m,n,m - 3(n-1)). 
\end{align*}
\endgroup
\textbf{Case 1(b):} $n+r$ is odd. Write $n+r = 2k+1$ for some integer $k \geq 3$. Consider the ideal $\tm_{[n-1]} (I_k^1)$. By Corollary \ref{cor:GrExamples}, this has the effect of transforming the associated tuple in the following way:
\begingroup\allowdisplaybreaks
\begin{align*}
    (2k+1, 1 , 2k+1) &\mapsto (2k+1+2(n-1)-1 ,1 + n-1 , 2k+1 - (n-1)-1) \\
    &= (n+r+2n-3, n ,n+r-n+1-1) \\
    &= (m,n,m - 3(n-1)). 
\end{align*}
\endgroup 
\textbf{Case 2(a):} $n+r$ is even. Write $n+r = 2k$ for some $k \geq 3$. Consider the ideal
$$\begin{cases}
    \tm_{1,2k+1} (I_k^2) & \textrm{if} \ n=3, \\
    \tm_{\{1,2k+1\} \cup [3,n-1]} (I_k^2 )& \textrm{if} \ n \geq 4. \\
\end{cases}$$
By Proposition \ref{prop:expQmapGr}, $\rank_k (Q_1 \otimes k) = 4$ and
$$\sigma \cap \{ j \mid \overline{Q_1 (f_2^j)} \neq 0 \} =
    \{ 1 , 2k+1 \}$$
so $|\sigma \cap \{ j \mid \overline{Q_1 (f_2^j)} \neq 0 \}| = 2$. By Corollary \ref{cor:GrExamples}, this has the effect of transforming the associated tuple as so:
\begingroup\allowdisplaybreaks
\begin{align*}
    (2k+1 , 1 , 2k+1) &\mapsto (2k+1 + 2(n-1) -4 , n , 2k+1 - (n-1) -4+2) \\
    &= (n+r + 2n - 5 , n , n+r-1 - (n-1))\\
    &= (m,n,r). 
\end{align*}
\endgroup
\textbf{Case 2(b):} $n+r$ is odd. Write $n+r = 2k+1$ for some $k \geq 3$. Consider the ideal
$$\begin{cases}
    \tm_{1,2k+1} (I_k^1) & \textrm{if} \ n=3, \\
    \tm_{\{1,2k_1\} \cup [3,n-1]} (I_k^1 )& \textrm{if} \ n \geq 4. \\
\end{cases}$$
By Proposition \ref{prop:expQmapGr}, $\rank_k (Q_1 \otimes k) = 3$, and exactly as in Case $2(a)$,
$$|\sigma \cap \{ j \mid \overline{Q_1 (f_2^j)} \neq 0 \}| =
   | \{ 1 , 2k+1 \}| = 2.$$
   By Corollary \ref{cor:GrExamples}, this has the effect of transforming the associated tuple as so:
\begingroup\allowdisplaybreaks
\begin{align*}
    (2k+1 , 1 , 2k+1) &\mapsto (2k+1 + 2(n-1) -3 , n , 2k+1 - (n-1) -3+2) \\
    &= (n+r + 2n - 5 , n , n+r-1 - (n-1))\\
    &= (m,n,r). 
\end{align*}
\endgroup
\textbf{Case 3(a):} $n+r$ is odd. Write $n+r = 2k+1$ for some $k \geq 3$. Consider the ideal $\tm_{[n-2], 2k+1} (I_k^2)$. By Proposition \ref{prop:expQmapGr}, $\rank_k (Q_1 \otimes k) = 4$, and (recalling that $n \geq 4$)
$$\sigma \cap \{ j \mid \overline{Q_1 (f_2^j)} \neq 0 \} =
    \{ 1 ,2, 2k+1 \},$$
so $|\sigma \cap \{ j \mid \overline{Q_1 (f_2^j)} \neq 0 \}| = 3$.  By Corollary \ref{cor:GrExamples}, this has the effect of transforming the associated tuple as so:
\begingroup\allowdisplaybreaks
\begin{align*}
    (2k+1 , 1 , 2k+1) &\mapsto (2k+1 + 2(n-1) -4 , n , 2k+1 - (n-1) -4+3) \\
    &= (n+r + 2n - 6 , n , n+r-1 - (n-1))\\
    &= (m,n,r). 
\end{align*}
\endgroup

\textbf{Case 3(b):} $n+r$ is even. Write $n+r = 2k+2$ for some $k \geq 3$. Consider the ideal $\tm_{[n-2], 2k+1} (I_k^1)$. By Proposition \ref{prop:expQmapGr}, $\rank_k (Q_1 \otimes k) = 3$, and (recalling that $n \geq 4$)
$$|\sigma \cap \{ j \mid \overline{Q_1 (f_2^j)} \neq 0 \}| =
    |\{ 1 ,2, 2k+1 \}|=3$$
By Corollary \ref{cor:GrExamples}, this has the effect of transforming the associated tuple as so:
\begingroup\allowdisplaybreaks
\begin{align*}
    (2k+1 , 1 , 2k+1) &\mapsto (2k+1 + 2(n-1) -3 , n , 2k+1 - (n-1) -3+3) \\
    &= (n+r + 2n - 6 , n , n+r-1 - (n-1))\\
    &= (m,n,r). 
\end{align*}
\endgroup
\end{proof}

\begin{example}
Corollary \ref{cor:classGrealizable} is far from being an exhaustive list of the possible tuples $(m,n,r)$. For example, let $R = k[x_1 , x_2 , x_3]$ and consider the ideal
$$J:= ({x}_{1}^{2}{x}_{3},{x}_{1}^{2}{x}_{2}-{x}_{3}^{3},{x}_{2}^{2}{x}_{3}^{2},{x}_{1}{x}_{2}^{2}{x}_{3},{x}_{2}^{4},{x}_{1}{x}_{2}^{3},{x}_{1}^{4} ).$$
One may verify using the TorAlgebras package in Macaulay2 that $J$ defines a ring of Tor-algebra class $G(2)$ and realizes the tuple $(7,2,2)$, which does not fall into any of the cases of Corollary \ref{cor:classGrealizable}.
\end{example}

As of yet, there is no standardized method for producing non-Gorenstein rings of Tor-algebra class $G(r)$ en masse besides trimming; because of this, the realizable classes covered by Corollary \ref{cor:classGrealizable} are bound to be rather restricted. Next, we consider rings of class $H(p,q)$. 

\begin{setup}\label{set:Hpqsetup}
Adopt notation and hypotheses as in Setup \ref{set:examplesSet}. Let $X_p$ denote the $p \times (p-1)$ matrix
$$X_p := \begin{pmatrix}
x_1 & 0 & 0 & \cdots & 0 \\
x_2 & x_1 & 0 & \cdots & 0\\
x_3 & x_2 & x_1 & \cdots &0 \\
0 & x_3 & x_2 & \cdots &0 \\
\vdots & \ddots & \ddots & \ddots &0 \\
0 & 0 & \cdots & \cdots & x_1 \\
0 & 0 & \cdots & \cdots & x_2 \\
\end{pmatrix}$$
and define 
$$J_p := I_{p-1} (X_p) + (x_3^{p-1}) = (\Delta_1 , \dots , \Delta_p , x_3^{p-1} ).$$ 
Let $H_\bullet$ denote the Hilbert-Burch resolution of $R/I_{p-1} (X_p)$ and $G_\bullet := 0 \to R \xrightarrow{x_3^{p-1}} R$. The minimal free resolution of $J_p$ may be obtained as the tensor product $F_\bullet := (H \otimes G)_\bullet$:
$$F_\bullet : \quad \xymatrix{0 \ar[r] & H_2 \otimes G_1 \ar[r] & \big( H_1 \otimes G_1 \big) \oplus H_2 \ar[r] & H_1 \oplus G_1 \ar[r] & R. \\}$$
The following multiplication makes $F_\bullet$ into an algebra resolution in standard form of Tor-algebra class $H(p,p-1)$:
\begingroup\allowdisplaybreaks
\begin{align*}
    h_1 \cdot_F h_1' &= h_1 \cdot_H h_1', \\
    h_1 \cdot_F g_1 &= h_1 \otimes g_1, \\
    h_2 \cdot_F g_1 &= h_2 \otimes g_1, \\
    \textrm{where} \ h_1, \ h_1' &\in H_1, h_2 \in H_2, g_1 \in G_1 . \\
\end{align*}
\endgroup
In an identical manner, let $F_\bullet'$ be a minimal algebra resolution of $J_p':= I_{p-1} (X_p') + (x_2^{2n-2})$ in standard from of Tor-algebra class $H(p,p-1)$, where
$$X_p' :=  \begin{pmatrix}
x_1^2 & 0 & 0 & \cdots & 0 \\
x_2^2 & x_1^2 & 0 & \cdots & 0\\
x_3^2 & x_2^2 & x_1^2 & \cdots &0 \\
0 & x_3^2 & x_2^2 & \cdots &0 \\
\vdots & \ddots & \ddots & \ddots &0 \\
0 & 0 & \cdots & \cdots & x_1^2 \\
0 & 0 & \cdots & \cdots & x_2^2 \\
\end{pmatrix}.$$
\end{setup}

\begin{prop}\label{prop:expQmapsHpq}
Adopt notation and hypotheses as in Setup \ref{set:Hpqsetup}. Assume $H_1 = \bigoplus_{i=1}^p Rh_1^i$, where $h_1^i \mapsto \Delta_i$. Define $q_1^i : F_2 \to K_1$ for $i<p$ by sending:
\begingroup\allowdisplaybreaks
\begin{align*}
    h_2^{i-2} &\mapsto e_3 \quad (i > 2), \\
    h_2^{i-1} & \mapsto e_2 \quad (i>1), \\
    h_2^i &\mapsto e_1 \quad (i<p), \\
    h_1^i \otimes g_1 &\mapsto -x_3^{p-2} e_3, \\
\end{align*}
\endgroup
and all other basis elements to $0$. If $i=p+1$, write each $\Delta_j = x_1 \Delta_{1,j} + x_2 \Delta_{2,j} + x_3 \Delta_{3,j}$. Define $q_1^{p+1} : F_2 \to K_1$ by sending:
\begingroup\allowdisplaybreaks
\begin{align*}
    h_1^j \otimes g_1 &\mapsto \Delta_{1,j} e_1 + \Delta_{2,j} e_2 + \Delta_{3,j} e_3, \quad (1 \leq j \leq p) \\
\end{align*}
\endgroup
and all other basis elements to $0$. Then the following diagram commutes:
$$\xymatrix{F_2 \ar[dr]^-{{d_0'}^i} \ar[d]_-{q_1^i} & \\
K_1 \ar[r]^-{m_1} & \m. \\}$$
\end{prop}

\begin{cor}\label{cor:classHpqEx}
Adopt notation and hypotheses as in Setup \ref{set:Hpqsetup} with $p \geq 4$. Then
\begin{enumerate}
    \item if $p+1 \notin \sigma$, the ideal $\tm_\sigma (J_p)$ defines either a Golod ring or a ring of Tor-algebra class $H(p-t , p-1 - \rank_k (Q_1 \otimes k)$. In particular, if $\sigma = [t]$ for some $t<p$, the ideal $\tm_{[t]} (J_p)$ defines a ring of Tor-algebra class $H(p-t,p-1-t)$. 
    \item if $p+1 \in \sigma$, the ideal $\tm_\sigma (I)$ defines a Golod ring.
    \item if $p+1 \notin \sigma$, the ideal $\tm_\sigma (J'_p)$ defines a ring of Tor-algebra class $H(p-t,p-1)$.
    \item if $p+1 \in \sigma$, the ideal $\tm_\sigma (J'_p)$ defines a Golod ring. 
\end{enumerate}
\end{cor}

\begin{proof}
As in the proof of Corollary \ref{cor:GrExamples}, one has
$$\ker (Q_1 \otimes k ) =\Span_k \{ \overline{f_2^j} \in \overline{F_2} \mid \overline{Q_1 (f_2^j)} = 0 \}.$$
Moreover, the assumption that $p \geq 4$ implies that the hypotheses of Lemma \ref{lem:classHpq} are satisfied.

\textbf{Case 1:} Since $\ker (Q_1 \otimes k)$ is obtained by simply deleting basis elements with nonzero image under $Q_1 \otimes k$, it follows that $\rank_k (\ker (Q_1 \otimes k) \to \hom_k (\overline{F_1'}, \overline{F_3} ) \leq r - \rank_k (Q_1 \otimes k)$. By Lemma \ref{lem:classHpq}, the result follows.

\textbf{Case 2:} This is simply case $(i)$ of Lemma \ref{lem:classHpq}.

\textbf{Cases 3 and 4:} Observe that $F_\bullet'$ has the property that $d_2 (F_2) \subseteq \m^2 F_1$. Thus the conclusion follows by Corollary \ref{cor:quadraticGr}.
\end{proof}

\begin{remark}
In terms of the tuples of the associated tuple (see Remark \ref{rk:theTuples}), the transformation $J_p \mapsto \tm_\sigma (J_p)$ transforms the tuple $(p+1,p-1,p,p-1)$ as so:
{\small
\begingroup\allowdisplaybreaks
\begin{align*}
    J_p &\mapsto \tm_\sigma (J_p) \\
    (p+1,p-1,p,p-1) &\mapsto (p+1+2t- \rank_k (Q_1 \otimes k) , p-1+t , p -t , p-1 - \rank_k (Q_1 \otimes k)), \\
    J_p' &\mapsto \tm_\sigma (J_p') \\
    (p+1,p-1,p,p-1) &\mapsto (p+1+2t , p-1+t , p -t , p-1).
\end{align*}
\endgroup}
\end{remark}

We conclude with some discussion on the problem of realizability. The results of Section \ref{sec:ToralgCons} are stated for arbitrary rings of a given Tor-algebra class. However, one must start with a ring of a given Tor-algebra class and then apply the trimming process to obtain a new ideal with some new set of parameters. The only simple candidates for ``initial" ideals of Tor-algebra class $G(r)$ and $H(p,q)$ are grade $3$ Gorenstein ideals and grade $3$ hyperplane sections, respectively. Even though using a combination of linkage and trimming can obtain \emph{many} of the tuples falling within the bounds of \cite[Theorem 1.1]{christensen2020linkage}, one is tempted to ask:

\begin{question}\label{question:otherSources}
Are there other ``canonical" sources of rings of Tor-algebra class $G(r)$ and $H(p,q)$, distinct from grade $3$ Gorenstein ideals or hyperplane sections?
\end{question}

Enlarging the set of starting ideals from which one can begin the process of linkage/trimming would immediately allow one to add to the question of realizability. As it turns out, rings of Tor-algebra class $G(r)$ arise generically when working in a polynomial ring. The examples arising in \cite{vandebogert2020resolution} are already obtained by trimming a Gorenstein ideal, but it is shown more generally in \cite{christensen2020generic} that generically obtained rings of type $2$ are of class $G(r)$ under appropriate hypotheses. To the author's knowledge, there are fewer results of this flavor for rings of Tor-algebra class $H(p,q)$, even though these rings seem to be ubiquitous.   

\section*{Acknowledgements}
Thank you to the anonymous referee for a very close reading and many suggestions that have greatly improved the paper.

\bibliographystyle{amsplain}
\bibliography{biblio}

\providecommand{\bysame}{\leavevmode\hbox to3em{\hrulefill}\thinspace}
\providecommand{\MR}{\relax\ifhmode\unskip\space\fi MR }
\providecommand{\MRhref}[2]{%
  \href{http://www.ams.org/mathscinet-getitem?mr=#1}{#2}
}
\providecommand{\href}[2]{#2}
\begin{thebibliography}{10}

\bibitem{avramov1998infinite}
Luchezar~L Avramov, \emph{Infinite free resolutions}, Six lectures on
  commutative algebra, Springer, 1998, pp.~1--118.

\bibitem{avramov2012cohomological}
\bysame, \emph{A cohomological study of local rings of embedding codepth 3},
  Journal of Pure and Applied Algebra \textbf{216} (2012), no.~11, 2489--2506.

\bibitem{avramov1988poincare}
Luchezar~L Avramov, Andrew~R Kustin, and Matthew Miller, \emph{Poincar{\'e}
  series of modules over local rings of small embedding codepth or small
  linking number}, Journal of Algebra \textbf{118} (1988), no.~1, 162--204.

\bibitem{buchsbaum1977algebra}
David~A Buchsbaum and David Eisenbud, \emph{Algebra structures for finite free
  resolutions, and some structure theorems for ideals of codimension 3},
  American Journal of Mathematics \textbf{99} (1977), no.~3, 447--485.

\bibitem{christensen2020generic}
Lars~Winther Christensen and Oana Veliche, \emph{Generic local artinian
  algebras of type 2}, In Preparation.

\bibitem{christensen2014local}
\bysame, \emph{Local rings of embedding codepth 3. examples}, Algebras and
  Representation Theory \textbf{17} (2014), no.~1, 121--135.

\bibitem{christensen2019trimming}
Lars~Winther Christensen, Oana Veliche, and Jerzy Weyman, \emph{Trimming a
  gorenstein ideal}, Journal of Commutative Algebra \textbf{11} (2019), no.~3,
  325--339.

\bibitem{christensen2020linkage}
\bysame, \emph{Linkage classes of grade 3 perfect ideals}, Journal of Pure and
  Applied Algebra \textbf{224} (2020), no.~6, 106185.

\bibitem{faucett2016expanding}
Jessica~Ann Faucett, \emph{Expanding the socle of a codimension 3 complete
  intersection}, Rocky Mountain Journal of Mathematics \textbf{46} (2016),
  no.~5, 1489--1498.

\bibitem{M2}
Daniel~R. Grayson and Michael~E. Stillman, \emph{Macaulay2, a software system
  for research in algebraic geometry}, Available at
  \texttt{http://www.math.uiuc.edu/Macaulay2/}.

\bibitem{vandebogert2019structure}
Keller VandeBogert, \emph{Structure theory for a class of grade 3 homogeneous
  ideals defining type 2 compressed rings}, arXiv preprint arXiv:1912.06949
  (2019).

\bibitem{vandebogert2020resolution}
\bysame, \emph{Resolution and tor algebra structures of grade 3 ideals defining
  compressed rings}, arXiv preprint arXiv:2004.06691 (2020).

\bibitem{vandebogert2020trimming}
\bysame, \emph{Trimming complexes and applications to resolutions of
  determinantal facet ideals}, Communications in Algebra (2020).

\bibitem{weyman1989structure}
Jerzy Weyman, \emph{On the structure of free resolutions of length 3}, Journal
  of Algebra \textbf{126} (1989), no.~1, 1--33.

\end{thebibliography}
\addcontentsline{toc}{section}{Bibliography}

\end{document}